\documentclass[a4paper]{amsart}
\usepackage{amsmath}
\usepackage{geometry,amsthm,graphics,tabularx,amssymb,color,verbatim}
\usepackage{amscd,CJKutf8}
\usepackage{bm,bbm}
\usepackage[all]{xypic}

\newcommand{\wh}{\widehat}
\newcommand{\Res}{\mathrm{Res}}

\newcommand{\ot}{\otimes}
\newcommand{\op}{\oplus}

\newcommand{\CL}{\mathcal{L}}
\newcommand{\CA}{\mathcal{A}}

\newcommand{\CR}{\mathcal{R}}
\newcommand{\g}{\mathfrak{g}}

\newcommand{\fgl}{\mathfrak{gl}}
\newcommand{\C}{\mathbb{C}}
\newcommand{\D}{\mathcal{D}}

\newcommand{\N}{\mathbb N}
\newcommand{\Z}{\mathbb Z}

\newcommand{\mk}{\mathbf k}
\newcommand{\mb}{\mathbf}
\newcommand{\Hom}{\mathrm{Hom}}
\newcommand{\nno}{\nonumber}

\newcommand{\pd}[2]{\({#1}^{#2}\frac{\partial}{\partial #1}\)}

\newcommand{\Del}[3]{\Delta_{#1}^{(#2)}(#3)}
\newcommand{\End}{\mathrm{End}}

\newcommand{\U}{\mathcal{U}}

\newcommand{\te}[1]{\textnormal{{#1}}}

\theoremstyle{Theorem}

\theoremstyle{Theorem}

\newtheorem{thm}{Theorem}[section]

\newtheorem{lemt}[thm]{Lemma}
\newtheorem{prpt}[thm]{Proposition}
\newtheorem{thmt}[thm]{Theorem}
\newtheorem{remt}[thm]{Remark}

\newtheorem{dfnt}[thm]{Definition}

\def\({\left(}

\def\){\right)}

\newlength{\dhatheight}

\def \<{{\langle}}
\def \>{{\rangle}}

\newcommand{\vac}{\mathbf{1}}
\newcommand{\E}{{\mathcal{E}}}

\allowdisplaybreaks[4]

\numberwithin{equation}{section}
\title[Quasi vertex Lie algebras and vertex algebras]{A unified construction of vertex algebras from infinite-dimensional Lie algebras}
\author{Fulin Chen$^1$}
\address{School of Mathematical Sciences, Xiamen University,
 Xiamen, China 361005} \email{chenf@xmu.edu.cn}
 \thanks{$^1$Partially supported by China NSF grant (No.11971397)}
 \author{Xiaoling Liao}\address{School of Mathematical Sciences, Xiamen University,
 Xiamen, China 361005} \email{xiaoling@stu.xmu.edu.cn}
 \author{Shaobin Tan$^2$}\address{School of Mathematical Sciences, Xiamen University,
 Xiamen, China 361005} \email{tans@xmu.edu.cn }\thanks{$^2$Partially supported by China NSF grants (No.12131018)}
\author{Qing Wang$^3$}\address{School of Mathematical Sciences, Xiamen University,
 Xiamen, China 361005} \email{qingwang@xmu.edu.cn }\thanks{$^3$Partially supported by
 China NSF grants (Nos.12071385, 12161141001, 11971396) and the Fundamental Research Funds for the Central Universities (No.20720200067)}

\subjclass[2010]{17B65 \& 17B69}
\keywords{Vertex algebra, infinite-dimensional Lie algebra, equivariant coordinated quasi module}
\begin{document}
\bibliographystyle{plain}

\begin{abstract}
In this paper, we give a unified construction of vertex algebras arising from infinite-dimensional Lie algebras, including the affine Kac-Moody algebras, Virasoro algebras, Heisenberg algebras and their  higher rank analogs, orbifolds and  deformations. We
 define a notion of what we call quasi vertex Lie algebra to unify these Lie algebras.
Starting from any (maximal) quasi vertex Lie algebra $\g$, we
construct a corresponding vertex Lie algebra $\g_0$,
and establish a canonical isomorphism between the category of restricted  $\g$-modules and that of
 equivariant $\phi$-coordinated quasi $V_{\g_0}$-modules, where $V_{\g_0}$ is the universal enveloping vertex algebra of $\g_0$.
 This unified all the previous constructions of vertex algebras from infinite-dimensional Lie algebras
 and shed light on the way to associate vertex algebras with Lie algebras.
\end{abstract}
\maketitle
\section{Introduction}
Vertex algebras and their modules are often constructed from restricted modules of infinite-dimensional Lie algebras such as the
(untwisted) affine Kac-Moody algebras, Heisenberg algebras and  Virasoro algebras (see \cite{DL,FZ,Li1}, etc.).
These (affine, Heisenberg and Virasoro) vertex  algebras are the major building blocks in vertex (operator) algebra theory.
Furthermore, their various higher rank analogs, orbifolds and  deformations, including the twisted affine Lie algebras \cite{K},
 (twisted and untwisted) toroidal extended affine Lie algebras \cite{ABFP,B}, quantum torus Lie algebras  \cite{BGK,G-KL},
 $q$-Heisenberg Lie algebras \cite{FR}, Virasoro-like  algebras \cite{LT}, Klein bottle Lie
algebras \cite{JJP,PR}, and $q$-Virasoro  algebras \cite{BC}, are studied extensively both in  mathematics and physics.
A fundamental problem, in the field of vertex algebras first formulated in \cite{Li-adv}, is to associate these Lie algebras with vertex algebras
 in the similar way to that
affine, Heisenberg and Virasoro  algebras are associated  with vertex algebras.

In this paper we first define a notion of what we call quasi vertex Lie algebra to unify
 the affine, Heisenberg and Virasoro algebras, as well as their various higher rank analogs, orbifolds and  deformations.
 Then we give an answer to this fundamental problem by associating
all quasi vertex Lie algebras
with vertex algebras in a unified way.

The notion of quasi vertex Lie algebra can be viewed as a ``non-commutative" generalization of
the vertex Lie algebra introduced by Dong-Li-Mason \cite{DLM} (cf. \cite{K2,P}).
We recall from \cite{DLM} that a vertex Lie algebra is a complex Lie algebra $\mathcal{L}$ together with
 a set $\mathcal{F}\subset \mathcal{L}[[z,z^{-1}]]$ of
generating functions on $\mathcal{L}$ that satisfies some axioms.
The main axiom is that  for any pair $(a(z),b(z))$ in $\mathcal{F}$,
\begin{align}\label{intro:commutator1}
[a(z),b(w)]=\sum_{i,j\ge 0}\frac{1}{i!}\(\(\frac{\partial}{\partial w}\)^jc_{i,j}(w)\)\left(\frac{\partial}{\partial w}\right)^i z^{-1}\delta\left(\frac{w}{z}\right)
\end{align}
for some $c_{i,j}(z)\in \mathcal{F}$. This axiom implies
that the commutator $[a(z),b(w)]$ is local in the sense that there exists a nonnegative integer $k$ such that
\begin{align*}
(z-w)^k [a(z),b(w)]=0.
\end{align*}
It was proved in \cite{DLM} (see also \cite{P}) that   there is   a  vertex algebra structure
on a distinguished  highest weight
$\mathcal{L}$-module $V_{\mathcal{L}}$, called the universal enveloping vertex algebra of $\mathcal{L}$ \cite{P},  such that
any restricted $\mathcal{L}$-module is naturally a $V_{\mathcal{L}}$-module.

We know that affine Lie algebras, Heisenberg algebras and Virasoro algebras are three most important families in the theory of infinite-dimensional Lie algebras, they are also the examples of vertex Lie algebras \cite{DLM}.
And almost all the Lie algebras we studied are related to these three families of Lie algebras, however, most of them
 are not vertex Lie algebras in general.
For example, the commutator of two generating functions of the untwisted toroidal extended affine Lie algebras or Virasoro-like algebras has the expression:  (cf. \cite{CLT,BLP})
\begin{align}\label{intro:commutator3}
\sum_{i,j\ge 0}\frac{1}{i!}\(\(w\frac{\partial}{\partial w}\)^jc_{i,j}(w)\)\left(w\frac{\partial}{\partial w}\right)^i \delta\left(\frac{w}{z}\right),
\end{align}
where $c_{i,j}(w)$ are some generating functions and we note that  degree  derivations other than  partial derivations are involved.
For some other Lie algebras, the  generating functions $a(z)$ and $b(z)$ are even not local.
They satisfies the so-called  quasi locality in the sense  that (cf. \cite{Li-adv})
\begin{align*}
 p(z,w) [a(z),b(w)]=0
\end{align*}
for some polynomial $p(z,w)$ (not necessarily has the form $(z-w)^k$). For example,
the  commutator of two  generating functions on the twisted affine Lie algebras, quantum $2$-torus Lie algebras or $q$-Virasoro algebras can be expressed as the following form:
(cf. \cite{Li3,LTW,GLTW1,GLTW2})
\begin{align}\label{intro:commutator2}
\sum_{m,n\in \Z}\sum_{i,j\geq0}\frac{1}{i!}\(\(\frac{\partial}{\partial w}\)^{j}c_{n,m,i,j}(q^m w)\)\(\frac{\partial}{\partial w}\)^{i}z^{-1}\delta\left(\frac{ q^n w}{ z}\right)\quad (q\in \C^\times),
\end{align}
where $c_{n,m,i,j}(w)$ are some  generating functions, while the commutator on the (nullity $2$) twisted toroidal extended affine Lie algebras,  $q$-Heisenberg Lie algebras or
Klein bottle Lie
algebras has the form: (cf. \cite{CTY,Li-FMC})
\begin{align}\label{intro:commutator4}
\sum_{m,n\in \Z}\sum_{i,j\geq0}\frac{1}{i!}\(\(w\frac{\partial}{\partial w}\)^{j}c_{n,m,i,j}(q^m w)\)\(w\frac{\partial}{\partial w}\)^{i}\delta\left(\frac{ q^n w}{ z}\right)\quad (q\in \C^\times),
\end{align}
where $c_{n,m,i,j}(w)$ are some  generating functions.

Motivated by these  commutator formulas \eqref{intro:commutator1}-\eqref{intro:commutator4}, we introduce the following definition.

\begin{dfnt}\label{def:g}
{\em
 A {\em quasi vertex Lie algebra} is a triple $(\g,\mathcal{A},\epsilon)$ consisting of a complex Lie algebra $\g$,
 a subset $\mathcal{A}$
 of $\g[[z,z^{-1}]]$, and an integer $\epsilon$, subject to  the following two conditions:
 \begin{itemize}
\item  $\g$ is linearly spanned by the coefficients of $a(z)\in \mathcal{A}$.
 \item for every pair  $(a(z),b(z))$ in $\mathcal{A}$, there is a (finite) subset
 \[\{(a_{(\alpha,\beta,i,j)}b)(z)\mid \alpha,\beta\in \C^\times, i,j\in \N\ \text{and all but finitely many $(a_{(\alpha,\beta,i,j)}b)(z)=0$}\}\]
 of $\mathcal{A}$ such that the commutator of $a(z)$ and $b(w)$ has the form:
\begin{equation}\label{[a,b]}\begin{split}
&[a(z),b(w)]=\sum_{\alpha,\beta\in \C^\times}\sum_{i,j\geq0}\frac{1}{i!}\(\(w^\epsilon\frac{\partial}{\partial w}\)^{j}(a_{(\alpha,\beta,i,j)}b)(\beta w)\)\(w^\epsilon\frac{\partial}{\partial w}\)^{i}z^{\epsilon-1}\delta\left(\frac{ \alpha w}{ z}\right).
\end{split}\end{equation}
\end{itemize}
We say that a $\g$-module $W$ is {\em restricted} if for any $a(z)\in \mathcal{A}$ and $w\in W$, $a(z)w\in W((z))$.}
\end{dfnt}

Note that the affine, Heisenberg, Virasoro algebras,
and their various higher rank analogs, orbifolds and  deformations mentioned above are all examples of  quasi vertex Lie algebras. In fact, quasi vertex Lie algebras include all the Lie algebras which have been associated to vertex algebras up to now.
For a quasi vertex Lie algebra $(\g,\mathcal{A},\epsilon)$, we define its associated group $\Gamma$  to be the subgroup of $\C^\times$ generated by
\[\{\alpha,\beta\in \C^\times\mid (a_{(\alpha,\beta,i,j)}b)(z)\ne 0\ \text{for some}\ a(z),b(z)\in \mathcal{A}\
\text{and}\ i,j\in \N\}.\]
We shall often denote the triple $(\g,\mathcal{A},\epsilon)$ simply by $\g$, and  write
\begin{align}
\mathcal{A}=\left\{a(z)=\sum_{m\in\Z}a(m)z^{-m+\epsilon-1}\mid a\in A\right\}\quad \text{($A$ is an index set).}
\end{align}

Let $\g=(\g,\mathcal{A},\epsilon)$ be a given quasi vertex Lie algebra and let $\Gamma$ be the associated group.
The main goal of this paper is to  construct a corresponding vertex algebra $V_{\g^0}$ such that
every restricted $\g$-module admits naturally certain $V_{\g^0}$-module structure.
Our main idea is to construct a sequence of Lie algebras $\g^\zeta$, $\zeta\in \Z$  based on $\g$ such that
\begin{itemize}
\item $\g^0$ is a vertex Lie algebra and so we have a corresponding vertex algebra $V_{\g^0}$;
\item    $\g$ can be ``reconstructed"
as a $\Gamma$-covariant algebra of $\g^\epsilon$, which allows us to associate $\g$  with $V_{\g^0}$ through its restricted modules.
\end{itemize}

More  precisely, for any integer $\zeta$,
let $\bar{\g}^{\zeta}$ denote the nonassociative algebra such that
\begin{itemize}
\item $\bar{\g}^{\zeta}$ admits a  basis
$\{\bar{a}^{\alpha,\zeta}(m)\mid a\in A, \alpha\in\Gamma,m\in\Z\}$;
\item the multiplication on the generating functions of $\bar{\g}^{\zeta}$ is given by
\begin{equation}\begin{split}\label{intro:multiplication}
&[\bar{a}^{\alpha,\zeta}(z),\bar{b}^{\beta,\zeta}(w)]\\
=&\alpha^{\epsilon-1}\sum_{\gamma\in \Gamma}\sum_{i,j\geq 0}\frac{1}{i!}\beta^{(i+j)(\epsilon-1)}\(\(w^\zeta\frac{\partial}{\partial w}\)^{j}\overline{a_{(\alpha\beta^{-1},\gamma,i,j)}b}^{\gamma\beta,\zeta}(w)\)\(w^\zeta\frac{\partial}{\partial w}\)^{i}z^{\zeta-1}\delta\(\frac{w}{z}\),
\end{split}\end{equation}
where $a,b\in A$, $\alpha,\beta\in \Gamma$ and
 $\bar{a}^{\alpha,\zeta}(z)=\sum_{m\in \Z} \bar{a}^{\alpha,\zeta}(m)z^{-m+\zeta-1}$.
 \end{itemize}
We denote by $\bar{\mathcal{A}}^{\zeta}$ the subspace of $\bar{\g}^{\zeta}[[z,z^{-1}]]$ spanned by the (linearly
independent) elements
$\(z^{\zeta}\frac{\partial}{\partial z}\)^{n}\bar{a}^{\alpha,\zeta}(z)$ for $n\in \N, a\in A, \alpha\in \Gamma$, and
define the linear map
\begin{align}\label{intro:psizeta}
\bar{\psi}^{\zeta}:\bar{\mathcal{A}}^{\zeta}\rightarrow \g[[z,z^{-1}]],\quad
\(z^{\zeta}\frac{\partial}{\partial z}\)^{n}\bar{a}^{\alpha,\zeta}(z)\mapsto \(z^{\epsilon}\frac{\partial}{\partial z}\)^{n}a(\alpha z).
\end{align}
Let $\bar{\g}^{\zeta}_0$ be the subspace of $\bar{\g}^{\zeta}$ spanned by all the coefficients
of the generating functions in $\ker \bar{\psi}^{\zeta}$, and
set
\[\g^\zeta=\bar{\g}^{\zeta}/\bar{\g}^{\zeta}_0.\]
For $a\in A, \alpha\in \Gamma$ and $m\in \Z$, we denote by $a^{\alpha,\zeta}(m)$ (resp.\,$a^{\alpha,\zeta}(z)$)  the image of $\bar{a}^{\alpha,\zeta}(m)$ (resp.\,$\bar{a}^{\alpha,\zeta}(z)$) in $\g^\zeta$
(resp.\,$\g^\zeta[[z,z^{-1}]]$).

We say that $\g$ is {\em maximal}
if as a $\C$-vector space,  $\g$ is abstractly spanned by the elements $a(m)$ for $a\in A$, $m\in \Z$, and subject to all (the coefficients of)
the  relations which have the form:
\[ \sum_{i\in I} \mu_i\(z^\epsilon\frac{\partial}{\partial z}\)^{n_i} a_i(\alpha_i z)=0,\]
where $I$ is a finite set, $\mu_i\in \C^\times, n_i\in \N, \alpha_i\in \Gamma$ and $a_i\in A$.

The first main theorem what we call  ``reconstruction theorem" of the paper is as follows, whose proof will be presented in Section 2.

\begin{thmt}\label{thm:main1} Let  $(\g,\CA,\epsilon)$ be a quasi vertex Lie algebra with   the associated group $\Gamma$, and let
$\zeta$ be an integer.
\begin{enumerate}
\item[(I)]
The space $\bar{\g}^\zeta_0$ is a two-sided ideal of the nonassociative algebra $\bar{\g}^\zeta$ and the quotient algebra
$\g^\zeta$ is a Lie algebra.
\item[(II)] When $\zeta=\epsilon$,  define a new multiplication $[\cdot,\cdot]_\Gamma$ on $\g^\epsilon$  by
\begin{align*}
[a^{\alpha,\epsilon}(m), b^{\beta,\epsilon}(n)]_{\Gamma}
=\sum_{\lambda\in \Gamma}\lambda^{-m+\epsilon-1}[a^{\alpha\lambda^{-1},\epsilon}(m), b^{\beta,\epsilon}(n)]
\end{align*}
for $a,b\in A, \alpha,\beta\in \Gamma$ and $m,n\in \Z$.
Then the subspace
\begin{align*}
\g^\epsilon_\Gamma=\te{Span}\{\lambda^{-m+\epsilon-1}a^{\alpha\lambda^{-1},\epsilon}(m)-a^{\alpha,\epsilon}(m)\mid
a\in A,\ \alpha, \lambda\in \Gamma,\ m\in \Z\}
\end{align*}
is a two-sided ideal of this new nonassociative algebra and the quotient algebra
$\g^\epsilon[\Gamma]=\g^\epsilon/\g^\epsilon_\Gamma$ is a Lie algebra.
Furthermore, the linear map
\begin{align*}
\varphi_{\g,\Gamma}:\g^\epsilon[\Gamma]\rightarrow \g,\quad a^{\alpha,\epsilon}(m)+\g^\epsilon_\Gamma\mapsto \alpha^{-m+\epsilon-1}a(m)\quad (a\in A,\alpha\in \Gamma, m\in \Z)
\end{align*}
is a surjective homomorphism, and $\varphi_{\g,\Gamma}$ is injective if and only if $\g$ is maximal.

\end{enumerate}
 \end{thmt}

For example, let $\mathfrak{b}$ be a Lie algebra equipped with an  automorphism $\sigma$ of order $T$.
Associated to the pair $(\mathfrak{b},\sigma)$, we have the twisted loop algebra \cite{K}
\[\mathcal{L}(\mathfrak{b},\sigma)=\prod_{k=0}^{T-1} \mathfrak{b}_{(k)}\ot t^k\C[t^T,t^{-T}]
\subset \mathfrak{b}\ot \C[t,t^{-1}],\]
 where $\mathfrak{b}_{(k)}=\{x\in \mathfrak{b}\mid
\sigma(x)=e^{2k\pi\sqrt{-1}/T} x\}$.
For any $\epsilon\in \Z$,  set
\[\mathcal{A}=\left\{a(z)=\sum_{m\in \Z} (a\ot t^{mT+k}) z^{-mT-k+\epsilon-1}\mid a\in \mathfrak{b}_{(k)}, k=0,\dots,T-1\right\}.\]
Then
 $(\mathcal{L}(\mathfrak{b},\sigma),\CA,\epsilon)$ is a maximal quasi vertex Lie algebra with the associated group $\Gamma=\<e^{2\pi\sqrt{-1}/T}\>$.
 Furthermore, we have $\mathcal{L}(\mathfrak{b},\sigma)^\zeta\cong \mathfrak{b}\ot \C[t,t^{-1}]$ for any $\zeta\in \Z$ and $\mathcal{L}(\mathfrak{b},\sigma)^\epsilon[\Gamma]\cong \mathcal{L}(\mathfrak{b},\sigma)$.

Now we recall the notion of the vertex algebra with a group action and its corresponding module (cf. \cite{Li-adv,JKLT,CLTW}).

\begin{dfnt}\label{def:geva}{\em (1) Let $\Gamma$ be a subgroup of $\C^\times$ and $\epsilon\in \Z$.
A {\em $(\Gamma,\epsilon)$-vertex algebra} $(V,Y,{\bm 1},R)$ is a vertex algebra $(V,Y,{\bm 1})$ together with
a group homomorphism
\[R:\Gamma\rightarrow \mathrm{GL}(V),\quad \alpha\mapsto R_\alpha\] such that for $\alpha\in \Gamma$ and $v\in V$,
\begin{align*}
R_\alpha({\bm 1})={\bm 1},\quad R_\alpha Y(v,z) R_\alpha^{-1}=Y(R_\alpha v, \alpha^{1-\epsilon}z).
\end{align*}
(2) Let $(V,Y,{\bm 1},R)$ be  a  $(\Gamma,\epsilon)$-vertex algebra and set
\[\phi_\epsilon(z,w)
=e^{w z^{\epsilon}\frac{d}{dz}}z\in \C((z))[[w]].\]
A  {\em $\Gamma$-equivariant $\phi_\epsilon$-coordinated quasi $V$-module} is a vector space
$W$ equipped with a linear map
$
Y_W^\epsilon(\cdot,z):V\rightarrow \mathrm{Hom}(W,W((z)))
$
satisfying the conditions that
\begin{itemize}
\item  $Y_{W}^\epsilon(\mathbf{1},z)=1_{W}$ (the identity map on $W$),
\item  $Y_W^\epsilon(R_\alpha v,z)=Y_W^\epsilon(v,\alpha^{-1}z)$ for $\alpha\in \Gamma$, $v\in V$,
\item  for $u,v\in V$, there exists $q(z)\in\C[z]$ whose roots lie in $\Gamma$ such that
\begin{eqnarray*}
&&q(z_1/z_2)Y_{W}^\epsilon(u,z_{1})Y_{W}^\epsilon(v,z_{2})\in\Hom(W,W((z_{1},z_{2}))),\\
&&q(\phi_\epsilon(z_{2},z_{0})/z_2) Y_{W}^\epsilon(Y(u,z_{0})v,z_{2})=(q(z_1/z_2) Y_{W}^\epsilon(u,z_{1})Y_{W}^\epsilon(v,z_{2}))\mid_{z_{1}=\phi_\epsilon(z_{2},z_{0})}.
\end{eqnarray*}
\end{itemize}
}
\end{dfnt}

When $\Gamma=\{1\}$ and $\epsilon=0$, a $(\Gamma,\epsilon)$-vertex algebra is a vertex algebra and its
$\Gamma$-equivariant $\phi_\epsilon$-coordinated quasi module is a usual module.
When $\epsilon=0$, the  $\Gamma$-equivariant $\phi_\epsilon$-coordinated quasi module is an equivariant quasi module introduced in \cite{Li-adv,Li3}.
When $\epsilon=1$, the $\Gamma$-equivariant $\phi_\epsilon$-coordinated quasi module was introduced in \cite{Li-JMP}. When $\Gamma=\{1\}$, this notion is the  $\phi_\epsilon$-coordinated module  introduced in \cite{Li-CMP} (cf. \cite{BLP}).

Let $\g=(\g,\CA,\epsilon)$ be a quasi vertex Lie algebra and let $\Gamma$ be the associated group.
 Recall from Theorem \ref{thm:main1} we have a sequence of Lie algebras $\g^\zeta$ for $\zeta\in\Z$. Particularly, we consider the Lie algebra $\g^0$, and set \begin{align}\label{intro:g0+}
\g^0_+=\text{Span}\{a^{\alpha,0}(n)\mid a\in A, \alpha\in \Gamma, n\ge 0\},
\end{align} which is a subalgebra of $\g^0$.
Form the induced $\g^0$-module
\begin{align}\label{intro:vg0}
V_{\g^0}= \U(\g^0)\otimes_{\U(\g^0_+)}\C,
\end{align}
where $\C$ denotes the trivial $\g^0_+$-module.
Set $\bm{1}=1\ot 1$ and $a^{\alpha,0}=a^{\alpha,0}(-1){\bm 1}\in V_{\g^0}$ for $a\in A$ and $\alpha\in \Gamma$.
The following is the second main result of this paper, whose proof will be given in Section 3.

\begin{thmt}\label{thm:main2}
Let $(\g,\mathcal{A},\epsilon)$ be a quasi vertex Lie algebra with $\Gamma$ the associated group.
\begin{enumerate}
\item[(I)] There is a unique $(\Gamma,\epsilon)$-vertex algebra structure on $V_{\g^0}$ such that $\bm{1}$ is the vacuum vector,
\begin{align*}
Y(a^{\alpha,0},z)=a^{\alpha,0}(z)\quad\te{and}\quad R_\lambda(a^{\alpha,0})=a^{\alpha\lambda^{-1},0}\end{align*}
for $a\in A$ and $\alpha,\lambda\in \Gamma$.
\item[(II)] Any restricted $\g$-module $W$ is naturally  a $\Gamma$-equivariant $\phi_\epsilon$-coordinated quasi $V_{\g^0}$-module with
the mapping $Y_W^\epsilon$ uniquely determined by
\begin{align*}
Y_W^{\epsilon}(a^{\alpha,0},z)=a(\alpha z)
\end{align*}
for $a\in A$ and $\alpha\in \Gamma$. Conversely, when $\g$ is maximal, any $\Gamma$-equivariant $\phi_\epsilon$-coordinated quasi $V_{\g^0}$-module
$(W,Y_W^\epsilon)$ is naturally a restricted $\g$-module $W$ with
\begin{align*}
a(z)=Y_W^{\epsilon}(a^{1,0},z)
\end{align*}
for $a\in A$.
\end{enumerate}
\end{thmt}

In literatures, the vertex algebras
$V_{\g^0}$ have been constructed from
 Lie algebras $\g$ such as twisted affine Lie algebras, toroidal extended affine Lie algebras, quantum $2$-torus Lie
algebras, $q$-Heisenberg Lie algebras, Virasoro-like algebras, $q$-Virasoro algebras and so on, and different approaches on associating these Lie algebras with vertex algebras are used (see \cite{Li3,Li-FMC,CLT,CTY,LTW,BLP,GLTW1,GLTW2}, etc.).
We emphasize a unified construction of vertex algebras arising from these Lie algebras and then we obtain correspondences between the restricted module categories of Lie algebras and certain quasi-module categories of vertex algebras.

In Section 4, we present five typical examples to show the applications of our main results:
(i) the twisted affine Lie algebras; (ii) the quantum $N+1$-torus Lie algebras; (iii) the $q$-Heisenberg Lie algebras;
(iv) the Virasoro-like algebras; (v) the Klein bottle Lie
algebras.
The examples (i), (ii) and (iii) have quasi vertex Lie algebra structures for any integer $\epsilon$,
 while the Lie algebras (iv) and (v) admit naturally  quasi vertex Lie algebra structures only when $\epsilon=1$.
Moreover, the examples (ii) (with $N\ge 2$) and (v) are new.

Let $\Z$, $\N$, $\C$ and $\C^\times$ be the set of integers, nonnegative integers, complex numbers and nonzero complex numbers,
respectively.  All the Lie algebras in this paper are over the field of complex numbers.
Let $z,w,z_{0},z_{1},z_{2},\dots$ be mutually
commuting independent formal variables.
For a  linear map  $\varphi:U\rightarrow W$ of vector spaces,
we will also write $\varphi$ for the linear map from the space $U[[z,z^{-1}]]$ of $U$-valued formal
Laurent series to $W[[z,z^{-1}]]$
such that  $\mathrm{Res}_z z^{m}\varphi(u(z))= \varphi(\mathrm{Res}_z z^m u(z))$ for  $m\in \Z$ and $u(z)\in U[[z,z^{-1}]]$.

\section{Proof of Theorem \ref{thm:main1}}
This section is devoted to the proof of  Theorem \ref{thm:main1}. Throughout this section, let $(\g,\CA,\epsilon)$ be a quasi vertex Lie algebra, let
$\Gamma$ be the group associated to $(\g,\CA,\epsilon)$, and let
 $\zeta$ be an integer as in  Theorem \ref{thm:main1}.

\subsection{Some lemmas}
We first establish some technical lemmas for later use.
For convenience, in the rest of the paper we will frequently use the following notation:
\begin{align*}
&\Del{ w,\zeta}{i}{\alpha z,\beta w}= \frac{1}{i!} \left(w^\zeta \frac{\partial}{\partial
  w}\right)^i \((\alpha z)^{\zeta-1}\delta\left(\frac{\beta w}{\alpha z}\right)\),\\
 &\Del{ z,\zeta}{i}{\alpha z,\beta w}= \frac{1}{i!} \left(z^\zeta \frac{\partial}{\partial
  z}\right)^i \((\alpha z)^{\zeta-1}\delta\left(\frac{\beta w}{\alpha z}\right)\),
\end{align*}
where $\alpha,\beta\in\C^\times$ and $i\in\N$. The following are some relations among these delta functions.

\begin{lemt}\label{lem:partial-delta}For any $\alpha,\beta\in\C^\times$ and $i\in \N$, we have
\begin{align}\label{eq:partial-delta0}
\Del{w,\zeta}{i}{\beta w,\alpha z}&=\Del{ w,\zeta}{i}{\alpha z,\beta w}=\alpha^{\zeta-1}\Del{ w,\zeta}{i}{ z,\alpha^{-1}\beta w},\\
\label{partial-delta1}\Del{ w,\zeta}{i}{\alpha z,\beta w}&=(-1)^i(\alpha\beta^{-1})^{i(\zeta-1)}\Del{ z,\zeta}{i}{\alpha z,\beta w}.
\end{align}
\end{lemt}
\begin{proof} \eqref{eq:partial-delta0} is straightforward to check.
For \eqref{partial-delta1}, from the fact
\begin{align*}
&\frac{\partial}{\partial  w}\(z^{-1} \delta\left(\frac{ w}{z}\right)\)
=-\frac{\partial}{\partial  z}\(z^{-1} \delta\left(\frac{ w}{z}\right)\),
\end{align*}
it follows that
\begin{align*}
\,& w^{\zeta} \frac{\partial}{\partial  w} \left((\alpha z)^{\zeta-1}\delta\left(\frac{\beta w}{\alpha z}\right)\right)
=\beta(\alpha zw)^{\zeta} \frac{\partial}{\partial \beta w} \((\alpha z)^{-1}\delta\left(\frac{\beta w}{\alpha z}\right)\)\\
=\, &-\beta(\alpha zw)^{\zeta} \frac{\partial}{\partial \alpha z} \left((\alpha z)^{-1}\delta\left(\frac{\beta w}{\alpha z}\right)\right)
=-(\alpha\beta^{-1})^{\zeta-1} z^{\zeta} \frac{\partial}{\partial  z} \left((\alpha z)^{\zeta-1}\delta\left(\frac{\beta w}{\alpha z}\right)\right).
\end{align*}
By an induction on $i$, this implies that
 \begin{align*}
&\Del{ w,\zeta}{i}{\alpha z,\beta w}=\frac{1}{i}w^{\zeta} \frac{\partial}{\partial  w}\Del{ w,\zeta}{i-1}{\alpha z,\beta w}\\
=\,&\frac{1}{i}w^{\zeta}\frac{\partial}{\partial  w}  \((-1)^{i-1}(\alpha\beta^{-1})^{(i-1)(\zeta-1)}\Del{ z,\zeta}{i-1}{\alpha z,\beta w}\)\\
=\,&\frac{1}{i!}(-1)^{i-1}(\alpha\beta^{-1})^{(i-1)(\zeta-1)}\left(z^\zeta \frac{\partial}{\partial
  z}\right)^{i-1}w^\zeta\frac{\partial}{\partial  w} \((\alpha z)^{\zeta-1}\delta\left(\frac{\beta w}{\alpha z}\right)\)\\
  =\,&(-1)^i(\alpha\beta^{-1})^{i(\zeta-1)}\Del{ z,\zeta}{i}{\alpha z,\beta w},
\end{align*}
as desired.
\end{proof}

 We will frequently use the following result without further explanation.

\begin{lemt}\label{delta-2} Let $\lambda_1,\dots,\lambda_r$ be distinct nonzero complex numbers, $k_1,\dots,k_r\in \N$ and
$A_{ij}(w)\in W[[w,w^{-1}]]$, where $W$ is a vector space, $i=1,\dots,r,$ and
$ 0\le j\le k_i$.
   Then
\begin{eqnarray}\label{eq:lemma220}
\sum _{i=1}^r\sum_{j=0}^{k_i} A_{ij}(w)\Del{w,\zeta}{j}{z,\lambda_i w}=0
\end{eqnarray}
if and only if $A_{ij}(w)=0$ for all $i,j$.
\end{lemt}
\begin{proof}  We only need to prove that if \eqref{eq:lemma220} holds, then $A_{ij}(w)=0$ for all $i,j$.
For every $s\in\N$, let $f_{s0}(w),f_{s1}(w),\dots,f_{ss}(w)$ be the polynomials (uniquely) determined by
$$\left(w^\zeta\frac{\partial}{\partial w}\right)^{s}=f_{ss}(w)\left(\frac{\partial}{\partial w}\right)^{s}+
\cdots+f_{s1}(w)\frac{\partial}{\partial w}+f_{s0}(w).$$
Then for any $i,j$, we have
\begin{align*}
A_{ij}(w)\Del{w,\zeta}{j}{z,\lambda_i w}
=z^{\zeta}\sum_{n=0}^j g_{ij,n}(w)\Del{w,0}{n}{z,\lambda_i w},
\end{align*}
where
\[g_{ij,n}(w)=\frac{n!}{j!}f_{jn}(w)A_{ij}(w).\]
By multiplying both sides of \eqref{eq:lemma220} with $z^{-\zeta}$, we obtain
$$\sum _{i=1}^r\sum_{n=0}^{k_i}\(\sum_{j=n}^{k_i} g_{ij,n}(w)\) \Del{w,0}{n}{z,\lambda_i w}=0.
$$
In view of \cite[Lemma 2.5]{Li-adv}, this implies that
\[\sum_{j=n}^{k_i} g_{ij,n}(w)=0\] for $1\le i\le r$ and $0\le n\le k_i$.
In particular, by taking $n=k_i$, we have
\[g_{ik_i,k_i}(w)=f_{k_ik_i}(w)A_{ik_i}(w)=0\] for $1\le i\le r$.
This shows that $A_{ik_i}(w)=0$ by noting that $f_{k_ik_i}(w)=w^{k_i\zeta}$.
Then the assertion follows from an induction  on $\max\{k_1,\dots,k_r\}$.
\end{proof}

The following result follows from the skew-symmetry of the Lie algebra $\g$.
\begin{lemt}\label{lem:skew-sym} For any $a,b\in A, \lambda\in \Gamma$ and $k\geq 0$, we have
\begin{equation}\label{skew-sym}\begin{split}
&\sum_{\gamma\in \Gamma}\sum_{j\geq0}\(w^\epsilon\frac{\partial}{\partial w}\)^{j}(a_{(\lambda^{-1},\gamma,k,j)}b)(\gamma w)\\
=\,&-\lambda^{(-k+1)(\epsilon-1)}\sum_{\gamma\in \Gamma}\sum_{j\geq0}\sum_{i=0}^j\frac{1}{i!}(-1)^{i+k}\lambda^{-j(\epsilon-1)}\(w^\epsilon\frac{\partial}{\partial w}\)^{j}(b_{(\lambda,\gamma,i+k,j-i)}a)(\lambda^{-1}\gamma w).
\end{split}\end{equation}
\end{lemt}
\begin{proof}  Recall from  \eqref{[a,b]} and \eqref{eq:partial-delta0} that
\begin{eqnarray*}\begin{split}
[a(z),b(w)]=&\sum_{\lambda,\gamma\in \Gamma}\sum_{j,k\geq0}\(w^\epsilon\frac{\partial}{\partial w}\)^{j}(a_{(\lambda,\gamma,k,j)}b)(\gamma w)\Del{w,\epsilon}{k}{ z,\lambda w}\\
=& \sum_{\lambda,\gamma\in \Gamma}\sum_{j,k\geq0}\(w^\epsilon\frac{\partial}{\partial w}\)^{j}(a_{(\lambda,\gamma,k,j)}b)(\gamma w)\lambda^{\epsilon-1}\Del{w,\epsilon}{k}{ w,\lambda^{-1} z}\\
=& \sum_{\lambda,\gamma\in \Gamma}\sum_{j,k\geq0}\(w^\epsilon\frac{\partial}{\partial w}\)^{j}(a_{(\lambda^{-1},\gamma,k,j)}b)(\gamma w)\lambda^{1-\epsilon}\Del{w,\epsilon}{k}{ w,\lambda z}.
\end{split}\end{eqnarray*}
On the other hand, from \eqref{[a,b]} and  Lemma \ref{lem:partial-delta}, it follows that
\begin{eqnarray*}
&&\quad[b( w),a( z)]=\sum_{\lambda,\gamma\in \Gamma}\sum_{i,j\geq0}\(\(z^\epsilon\frac{\partial}{\partial z}\)^{j}(b_{(\lambda,\gamma,i,j)}a)(\gamma z)\)\Del{z,\epsilon}{i}{ w,\lambda z}\\
&&=\sum_{\lambda,\gamma\in \Gamma}\sum_{i,j\geq0}(-1)^i\lambda^{-i(\epsilon-1)}\(\(z^\epsilon\frac{\partial}{\partial z}\)^{j}(b_{(\lambda,\gamma,i,j)}a)(\gamma z)\)\Del{w,\epsilon}{i}{ w,\lambda z}\\
&&=\sum_{\lambda,\gamma\in \Gamma}\sum_{i,j\geq0}\frac{(-1)^i}{i!}\lambda^{-(i+j)(\epsilon-1)}\(w^\epsilon\frac{\partial}{\partial w}\)^i\(\(\(w^\epsilon\frac{\partial}{\partial w}\)^{j}(b_{(\lambda,\gamma,i,j)}a)(\lambda^{-1}\gamma w)\)\Del{w,\epsilon}{0}{ w,\lambda z}\)\\
&&=\sum_{\lambda,\gamma\in \Gamma}\sum_{i,j\geq0}\sum_{k=0}^{i}\frac{(-1)^i}{(i-k)!}\lambda^{-(i+j)(\epsilon-1)}\(\(w^\epsilon\frac{\partial}{\partial w}\)^{i+j-k}(b_{(\lambda,\gamma,i,j)}a)(\lambda^{-1}\gamma w)\)\Del{w,\epsilon}{k}{ w,\lambda z}\\
&&=\sum_{\lambda,\gamma\in \Gamma}\sum_{i,j,k\geq0}\frac{(-1)^{i+k}}{i!}\lambda^{-(i+j+k)(\epsilon-1)}\(\(w^\epsilon\frac{\partial}{\partial w}\)^{i+j}(b_{(\lambda,\gamma,i+k,j)}a)(\lambda^{-1}\gamma w)\)\Del{w,\epsilon}{k}{ w,\lambda z}\\
&&=\sum_{\lambda,\gamma\in \Gamma}\sum_{j,k\geq0}\sum_{i=0}^j\frac{(-1)^{i+k}}{i!}\lambda^{-(j+k)(\epsilon-1)}\(\(w^\epsilon\frac{\partial}{\partial w}\)^{j}(b_{(\lambda,\gamma,i+k,j-i)}a)(\lambda^{-1}\gamma w)\)\Del{w,\epsilon}{k}{ w,\lambda z}.
\end{eqnarray*}

In view of  Lemma \ref{delta-2} and the skew-symmetry $[a( z),b( w)]=-[b( w),a( z)]$, we obtain
 the equation \eqref{skew-sym}  by comparing the coefficients of $\Del{w,\epsilon}{k}{ w,\lambda z}$ in the above two equations.
\end{proof}

Due to the Jacobi identity of the Lie algebra $\g$, we have the following result.

\begin{lemt}For any $a,b,c\in A$, $\lambda,\eta\in \Gamma$ and $i,k\in\N$, we have

\begin{equation}\label{jacobi-quasi}\begin{split}
&\sum_{\substack{\xi,\gamma\in \Gamma\\l\geq0}}\sum_{s=0}^i\sum_{j=0}^l\binom{j+s}{s}\frac{i!}{(i-s)!}\xi^{(i+l-s-j)(\epsilon-1)}\(z^\epsilon\frac{\partial}{\partial z}\)^{l}\(a_{(\eta\xi^{-1},\gamma,i-s,l-j)}(b_{(\lambda,\xi,k,j+s)}c)\)(\gamma\xi z)\\
=&\sum_{\substack{\xi,\gamma\in \Gamma\\l\geq0}}\sum_{s=0}^i\sum_{j=0}^{k+s}\binom{i}{s}\frac{(-1)^{j}(k+s)!}{(k+s-j)!}\xi^{\epsilon-1}\lambda^{(i+j-s)(\epsilon-1)}\(z^\epsilon\frac{\partial}{\partial z}\)^{l}\((a_{(\eta\lambda^{-1},\xi,i-s,j)}b)_{(\lambda\xi,\gamma,k+s-j,l)}c\)(\gamma z)\\
 &+\sum_{\substack{\xi,\gamma\in \Gamma\\l\geq0}}\sum_{s=0}^k\sum_{j=0}^l\binom{j+s}{s}\frac{k!}{(k-s)!}\xi^{(k+l-s-j)(\epsilon-1)}\(z^\epsilon\frac{\partial}{\partial z}\)^{l}\(b_{(\lambda\xi^{-1},\gamma,k-s,l-j)}(a_{(\eta,\xi,i,j+s)}c)\)(\gamma\xi z).
\end{split}\end{equation}
\end{lemt}
\begin{proof} We will prove the lemma by comparing the two-sides of the Jacobi identity
\begin{equation*}\begin{split}
&[a( z_1),[b(z_2),c(z_3)]]=[[a(z_1),b(z_2)],c(z_3)]+[b(z_2),[a(z_1),c(z_3)]]\quad (a,b,c\in A).
\end{split}\end{equation*}
By \eqref{[a,b]} and  Lemma \ref{lem:partial-delta}, we have
\begin{eqnarray*}
&&\quad[[a( z_1),b(z_2)],c(z_3)]\\
&&=\sum_{\eta,\xi\in \Gamma}\sum_{i,j\geq0}\(\(z_2^\epsilon\frac{\partial}{\partial z_2}\)^{j}[(a_{(\eta,\xi,i,j)}b)(\xi z_2),c(z_3)]\)\Del{z_2,\epsilon}{i}{ z_1,\eta z_2}\\
&&=\sum_{\substack{\eta,\xi,\lambda,\gamma\in \Gamma\\i,j,k,l\geq0}}\(\(z_3^\epsilon\frac{\partial}{\partial z_3}\)^{l}\((a_{(\eta,\xi,i,j)}b)_{(\lambda,\gamma,k,l)}c\)(\gamma z_3)\)\\
&&\quad\(\(z_2^\epsilon\frac{\partial}{\partial z_2}\)^{j}\Del{z_3,\epsilon}{k}{\xi z_2,\lambda z_3}\)\Del{z_2,\epsilon}{i}{ z_1,\eta z_2}\\
&&=\sum_{\substack{\eta,\xi,\lambda,\gamma\in \Gamma\\i,j,k,l\geq0}}(-1)^j\frac{(k+j)!}{k!}(\lambda\xi^{-1})^{j(\epsilon-1)}\(\(z_3^\epsilon\frac{\partial}{\partial z_3}\)^{l}\((a_{(\eta,\xi,i,j)}b)_{(\lambda,\gamma,k,l)}c\)(\gamma z_3)\)\\
&&\quad\Del{z_3,\epsilon}{k+j}{\xi z_2,\lambda z_3}\Del{z_2,\epsilon}{i}{ z_1,\eta z_2}\\
&&=\sum_{\substack{\eta,\xi,\lambda,\gamma\in \Gamma\\i,j,k,l\geq0}}\frac{(-1)^{i+j}}{i!k!}\eta^{-i(\epsilon-1)}(\lambda\xi^{-1})^{j(\epsilon-1)}\(\(z_3^\epsilon\frac{\partial}{\partial z_3}\)^{l}\((a_{(\eta,\xi,i,j)}b)_{(\lambda,\gamma,k,l)}c\)(\gamma z_3)\)\\
&&\quad\(z_3^\epsilon\frac{\partial}{\partial z_3}\)^{k+j}\(z_1^\epsilon\frac{\partial}{\partial z_1}\)^i\(\Del{z_3,\epsilon}{0}{\xi z_2,\lambda z_3}\Del{z_1,\epsilon}{0}{ z_1,\eta\lambda\xi^{-1} z_3}\)\\
&&=\sum_{\substack{\eta,\xi,\lambda,\gamma\in \Gamma\\i,j,k,l\geq0}}\sum_{s=0}^{k+j}(-1)^{j}\binom{i+s}{s}\frac{(k+j)!}{k!}(\lambda\xi^{-1})^{(i+j)(\epsilon-1)}\cdot\\
&&\quad\(\(z_3^\epsilon\frac{\partial}{\partial z_3}\)^{l}\((a_{(\eta,\xi,i,j)}b)_{(\lambda,\gamma,k,l)}c\)(\gamma z_3)\)
\Del{z_3,\epsilon}{k+j-s}{\xi z_2,\lambda z_3}\Del{z_3,\epsilon}{i+s}{ z_1,\eta\lambda\xi^{-1} z_3}\\
&&=\sum_{\substack{\eta,\xi,\lambda,\gamma\in \Gamma\\i,k,l\geq0}}\sum_{j=0}^k\sum_{s=0}^{k}(-1)^{j}\binom{i+s}{s}\frac{k!}{(k-j)!}(\lambda\xi^{-1})^{(i+j)(\epsilon-1)}\cdot\\
&&\quad\(\(z_3^\epsilon\frac{\partial}{\partial z_3}\)^{l}\((a_{(\eta,\xi,i,j)}b)_{(\lambda,\gamma,k-j,l)}c\)(\gamma z_3)\)
\Del{z_3,\epsilon}{k-s}{\xi z_2,\lambda z_3}\Del{z_3,\epsilon}{i+s}{ z_1,\eta\lambda\xi^{-1} z_3}\\
&&=\sum_{\substack{\eta,\xi,\lambda,\gamma\in \Gamma\\i,k,l,s\geq0}}\sum_{j=0}^{k+s}\binom{i+s}{s}\frac{(-1)^{j}(k+s)!}{(k+s-j)!}(\lambda\xi^{-1})^{(i+j)(\epsilon-1)}\cdot\\
&&\quad\(\(z_3^\epsilon\frac{\partial}{\partial z_3}\)^{l}\((a_{(\eta,\xi,i,j)}b)_{(\lambda,\gamma,k+s-j,l)}c\)(\gamma z_3)\)
\Del{z_3,\epsilon}{k}{\xi z_2,\lambda z_3}\Del{z_3,\epsilon}{i+s}{ z_1,\eta\lambda\xi^{-1} z_3}\\
&&=\sum_{\substack{\eta,\xi,\lambda,\gamma\in \Gamma\\i,k,l\geq0}}\sum_{s=0}^i\sum_{j=0}^{k+s}\binom{i}{s}\frac{(-1)^{j}(k+s)!}{(k+s-j)!}(\lambda\xi^{-1})^{(i+j-s)(\epsilon-1)}\cdot\\
&&\quad\(\(z_3^\epsilon\frac{\partial}{\partial z_3}\)^{l}\((a_{(\eta,\xi,i-s,j)}b)_{(\lambda,\gamma,k+s-j,l)}c\)(\gamma z_3)\)
 \Del{z_3,\epsilon}{k}{\xi z_2,\lambda z_3}\Del{z_3,\epsilon}{i}{ z_1,\eta\lambda\xi^{-1} z_3}\\
 &&=\sum_{\substack{\eta,\xi,\lambda,\gamma\in \Gamma\\i,k,l\geq0}}\sum_{s=0}^i\sum_{j=0}^{k+s}\binom{i}{s}\frac{(-1)^{j}(k+s)!}{(k+s-j)!}\xi^{\epsilon-1}\lambda^{(i+j-s)(\epsilon-1)}\cdot\\
&&\quad\(\(z_3^\epsilon\frac{\partial}{\partial z_3}\)^{l}\((a_{(\eta\lambda^{-1},\xi,i-s,j)}b)_{(\lambda\xi,\gamma,k+s-j,l)}c\)(\gamma z_3)\)
 \Del{z_3,\epsilon}{k}{ z_2,\lambda z_3}\Del{z_3,\epsilon}{i}{ z_1,\eta z_3}.
\end{eqnarray*}

On the other hand, from  \eqref{[a,b]} and  \eqref{eq:partial-delta0} we have
\begin{eqnarray*}
&&\quad[a(z_1),[b(z_2),c(z_3)]]\\
&&=\sum_{\eta,\xi\in \Gamma}\sum_{i,j\geq0}\(\(z_3^\epsilon\frac{\partial}{\partial z_3}\)^{j}[a( z_1),(b_{(\eta,\xi,i,j)}c)(\xi z_3)]\)\Del{z_3,\epsilon}{i}{ z_2,\eta z_3}\\
&&=\sum_{\substack{\eta,\xi,\lambda,\gamma\in \Gamma\\i,j,k,l\geq0}}\xi^{(k+l)(\epsilon-1)}\(z_3^\epsilon\frac{\partial}{\partial z_3}\)^{j}\Bigg(\bigg(\(z_3^\epsilon\frac{\partial}{\partial z_3}\)^l\(a_{(\lambda,\gamma,k,l)}(b_{(\eta,\xi,i,j)}c)\)(\gamma\xi z_3)\bigg)\\
&&\quad\quad\quad\Del{z_3,\epsilon}{k}{ z_1,\lambda \xi z_3}\Bigg)\Del{z_3,\epsilon}{i}{ z_2,\eta z_3}\\
&&=\sum_{\substack{\eta,\xi,\lambda,\gamma\in \Gamma\\i,j,k,l\geq0}}\sum_{s=0}^j\binom{j}{s}\frac{(k+s)!}{k!}\xi^{(k+l)(\epsilon-1)}\(\(z_3^\epsilon\frac{\partial}{\partial z_3}\)^{l+j-s}\(a_{(\lambda,\gamma,k,l)}(b_{(\eta,\xi,i,j)}c)\)(\gamma\xi z_3)\)\\
&&\quad\Del{z_3,\epsilon}{k+s}{ z_1,\lambda \xi z_3}\Del{z_3,\epsilon}{i}{ z_2,\eta z_3}\\
&&=\sum_{\substack{\eta,\xi,\lambda,\gamma\in \Gamma\\i,j,k,l,s\geq0}}\binom{j+s}{s}\frac{(k+s)!}{k!}\xi^{(k+l)(\epsilon-1)}\(\(z_3^\epsilon\frac{\partial}{\partial z_3}\)^{l+j}\(a_{(\lambda,\gamma,k,l)}(b_{(\eta,\xi,i,j+s)}c)\)(\gamma\xi z_3)\)\\
&&\quad\Del{z_3,\epsilon}{k+s}{ z_1,\lambda \xi z_3}\Del{z_3,\epsilon}{i}{ z_2,\eta z_3}\\
&&=\sum_{\substack{\eta,\xi,\lambda,\gamma\in \Gamma\\i,k,l\geq0}}\sum_{s=0}^k\sum_{j=0}^l\binom{j+s}{s}\frac{k!}{(k-s)!}\xi^{(k+l-s-j)(\epsilon-1)}\cdot\\
&&\quad\(\(z_3^\epsilon\frac{\partial}{\partial z_3}\)^{l}\(a_{(\lambda,\gamma,k-s,l-j)}(b_{(\eta,\xi,i,j+s)}c)\)(\gamma\xi z_3)\)\Del{z_3,\epsilon}{k}{ z_1,\lambda \xi z_3}\Del{z_3,\epsilon}{i}{ z_2,\eta z_3}\\
&&=\sum_{\substack{\eta,\xi,\lambda,\gamma\in \Gamma\\i,k,l\geq0}}\sum_{s=0}^k\sum_{j=0}^l\binom{j+s}{s}\frac{k!}{(k-s)!}\xi^{(k+l-s-j)(\epsilon-1)}\cdot\\
&&\quad\(\(z_3^\epsilon\frac{\partial}{\partial z_3}\)^{l}\(a_{(\lambda\xi^{-1},\gamma,k-s,l-j)}(b_{(\eta,\xi,i,j+s)}c)\)(\gamma\xi z_3)\)\Del{z_3,\epsilon}{k}{ z_1,\lambda  z_3}\Del{z_3,\epsilon}{i}{ z_2,\eta z_3}\\
&&=\sum_{\substack{\eta,\xi,\lambda,\gamma\in \Gamma\\i,k,l\geq0}}\sum_{s=0}^i\sum_{j=0}^l\binom{j+s}{s}\frac{i!}{(i-s)!}\xi^{(i+l-s-j)(\epsilon-1)}\cdot\\
&&\quad\(\(z_3^\epsilon\frac{\partial}{\partial z_3}\)^{l}\(a_{(\eta\xi^{-1},\gamma,i-s,l-j)}(b_{(\lambda,\xi,k,j+s)}c)\)(\gamma\xi z_3)\)\Del{z_3,\epsilon}{i}{ z_1,\eta  z_3}\Del{z_3,\epsilon}{k}{ z_2,\lambda z_3}.
\end{eqnarray*}

Similarly, we have
\begin{eqnarray*}
&&\quad[b(z_2),[a(z_1),c(z_3)]]\\
&&=\sum_{\substack{\eta,\xi,\lambda,\gamma\in \Gamma\\i,k,l\geq0}}\sum_{s=0}^k\sum_{j=0}^l\binom{j+s}{s}\frac{k!}{(k-s)!}\xi^{(k+l-s-j)(\epsilon-1)}\cdot\\
&&\quad\(\(z_3^\epsilon\frac{\partial}{\partial z_3}\)^{l}\(b_{(\lambda\xi^{-1},\gamma,k-s,l-j)}(a_{(\eta,\xi,i,j+s)}c)\)(\gamma\xi z_3)\)
\Del{z_3,\epsilon}{k}{ z_2,\lambda z_3}\Del{z_3,\epsilon}{i}{ z_1,\eta z_3}.
\end{eqnarray*}

Comparing the coefficients of  $\Del{z_3,\epsilon}{k}{ z_2,\lambda z_3}\Del{z_3,\epsilon}{i}{ z_1,\eta z_3}$ in the above three identities, we obtain
 \eqref{jacobi-quasi}.
\end{proof}

\subsection{Proof of Theorem \ref{thm:main1} (I)} For the first part of Theorem \ref{thm:main1},
recall  that
\[\bar{\g}^\zeta=\bigoplus_{a\in A, \alpha\in \Gamma, m\in \Z}\C \bar{a}^{\alpha,\zeta}(m)\]
is a nonassociative algebra with the multiplication given by \eqref{intro:multiplication}.
Recall also that $\bar{\g}^\zeta_0$ is the subspace of $\bar{\g}^\zeta$ spanned by the
coefficients of the generating functions in $\ker \bar{\psi}^\zeta$ (see \eqref{intro:psizeta}).

\begin{lemt}\label{lem:ideal} $\bar{\g}_0^\zeta$ is a two-sided ideal of the nonassociative algebra $\bar{\g}^\zeta$.
\end{lemt}
\begin{proof} Note that for $a,b\in A$ and $\alpha,\beta\in \Gamma$, we have from the Lie relation \eqref{[a,b]} that
\begin{equation}\label{[a(alpha z),b(beta w)]}\begin{split}
&[a(\alpha z),b(\beta w)]=\sum_{\lambda,\gamma\in \Gamma}\sum_{i,j\geq0}\beta^{(i+j)(\epsilon-1)}\(\(w^\epsilon\frac{\partial}{\partial w}\)^{j}(a_{(\lambda,\gamma,i,j)}b)(\gamma\beta w)\)\Del{w,\epsilon}{i}{\alpha z,\lambda\beta w}.
\end{split}\end{equation}
We fix a generating function
\[u(z)=\sum_{k=1}^l\mu_k\pd{z}{\zeta}^{n_k}\overline{a_k}^{\alpha_k,\zeta}(z)\in \bar{\mathcal{A}}^{\zeta}\subset \bar\g^\zeta[[z,z^{-1}]],\]
where $\mu_k\in\C,\ n_k\in\N,\ a_k\in A$ and $\alpha_k\in\Gamma$.
Then it follows from  \eqref{[a(alpha z),b(beta w)]} and \eqref{partial-delta1} that
\begin{align}\label{eq:uzbw1}
&\,[\bar{\psi}^\zeta(u(z)),b(\beta w)]
=\,\sum_{k=1}^l\mu_k\pd{z}{\epsilon}^{n_k}\( [a_k(\alpha_k z),b(\beta w)]\)\\
=&\,\sum_{k=1}^l\sum_{\lambda,\gamma\in \Gamma}\sum_{i,j\geq0}\mu_k\beta^{(i+j)(\epsilon-1)}\frac{(i+n_k)!}{i!}(-1)^{n_k}(\lambda\beta\alpha_k^{-1})^{n_k(\epsilon-1)}\notag\\
&\,\(\(w^\epsilon\frac{\partial}{\partial w}\)^{j}({a_k}_{(\lambda,\gamma,i,j)}b)(\gamma\beta w)\)\Del{w,\epsilon}{i+n_k}{\alpha_k z,\lambda\beta w}\notag\\
=&\,\sum_{i\geq0}\sum_{{\substack{1\le k\le l\\n_k\le i}}}\sum_{{\substack{\gamma\in \Gamma\\j\geq 0}}} \beta^{(i+j-n_k)(\epsilon-1)}\frac{i!(-1)^{n_k}\mu_k\alpha_k^{\epsilon-1}}{(i-n_k)!}\(\(w^\epsilon\frac{\partial}{\partial w}\)^{j}({a_k}_{(\alpha_k\beta^{-1},\gamma,i-n_k,j)}b)(\gamma\beta w)\)\notag\\
&\,\Del{w,\epsilon}{i}{ z, w}+\sum_{k=1}^l\sum_{{\substack{\lambda,\gamma\in \Gamma\\ \lambda\neq \alpha_k\beta^{-1}}}}\sum_{i,j\geq0}\mu_k\beta^{(i+j)(\epsilon-1)}\frac{(i+n_k)!}{i!}(-1)^{n_k}(\lambda\beta\alpha_k^{-1})^{n_k(\epsilon-1)}\notag\cdot\\
&\,\(\(w^\epsilon\frac{\partial}{\partial w}\)^{j}({a_k}_{(\lambda,\gamma,i,j)}b)(\gamma\beta w)\)\Del{w,\epsilon}{i+n_k}{\alpha_k z,\lambda\beta w}.\notag
\end{align}

On the other hand, in $\bar{\g}^\zeta[[z,z^{-1}]]$, we have
\begin{align}\label{eq:uzbw2}
&\qquad [u(z),\bar{b}^{\beta,\zeta}(w)]
=\sum_{k=1}^l\mu_k\pd{z}{\zeta}^{n_k}\([\overline{a_k}^{\alpha_k,\zeta}(z),\bar{b}^{\beta,\zeta}(w)]\)\\
=&\,\sum_{k=1}^l\sum_{\substack{\gamma\in\Gamma\\i,j\geq0}}(-1)^{n_k}\mu_k\alpha_k^{\epsilon-1}\beta^{(i+j)(\epsilon-1)}\frac{(i+n_k)!}{i!}\(\(w^\zeta\frac{\partial}{\partial w}\)^{j}\overline{{a_k}_{(\alpha_k\beta^{-1},\gamma,i,j)}b}^{\gamma\beta,\zeta}(w)\)\Del{w,\zeta}{i+n_k}{ z,w}\notag\\
=&\,\sum_{i\geq0}\sum_{{\substack{1\le k\le l\\n_k\le i}}}\sum_{\substack{\gamma\in\Gamma\\j\geq0}}\beta^{(i+j-n_k)(\epsilon-1)}\frac{i!(-1)^{n_k}\mu_k\alpha_k^{\epsilon-1}}{(i-n_k)!}\(\(w^\zeta\frac{\partial}{\partial w}\)^{j}\overline{{a_k}_{(\alpha_k\beta^{-1},\gamma,i-n_k,j)}b}^{\gamma\beta,\zeta}(w)\)\Del{w,\zeta}{i}{ z,w}.\notag
\end{align}

Assume  that $u(z)\in \ker \bar{\psi}^\zeta$. Since $[\bar{\psi}^\zeta(u(z)),b(\beta w)]=0$, it follows from \eqref{eq:uzbw1} that
\[\sum_{{\substack{1\le k\le l\\n_k\le i}}}\sum_{{\substack{\gamma\in \Gamma\\j\geq 0}}} \beta^{(i+j-n_k)(\epsilon-1)}\frac{i!(-1)^{n_k}\mu_k\alpha_k^{\epsilon-1}}{(i-n_k)!}\(w^\epsilon\frac{\partial}{\partial w}\)^{j}({a_k}_{(\alpha_k\beta^{-1},\gamma,i-n_k,j)}b)(\gamma\beta w)=0\]
for any $i\in\N$. This is equivalent to say that
\begin{equation*}\begin{split}
\sum_{{\substack{1\le k\le l\\n_k\le i}}}\sum_{\substack{\gamma\in\Gamma\\j\geq0}}\beta^{(i+j-n_k)(\epsilon-1)}\frac{i!(-1)^{n_k}\mu_k\alpha_k^{\epsilon-1}}{(i-n_k)!}\(w^\zeta\frac{\partial}{\partial w}\)^{j}\overline{{a_k}_{(\alpha_k\beta^{-1},\gamma,i-n_k,j)}b}^{\gamma\beta,\zeta}(w)\in\ker\bar{\psi}^\zeta.
\end{split}\end{equation*}
Thus  from \eqref{eq:uzbw2} we obtain that $\bar{\g}_0^\zeta$ is a right ideal of  $\bar{\g}^\zeta$.
Similarly, one can check that $\bar{\g}_0^\zeta$ is also a left ideal.
\end{proof}

Recall that for $a\in A, \alpha\in \Gamma$ and $m\in \Z$, $a^{\alpha,\zeta}(m)$ denotes the image of $\bar{a}^{\alpha,\zeta}(m)$ in
the quotient space $\g^\zeta=\bar{\g}^\zeta/\bar{\g}_0^\zeta$, and that
\begin{align}
a^{\alpha,\zeta}(z)=\sum_{m\in \Z}a^{\alpha,\zeta}(m) z^{-m+\zeta-1}\in \g^\zeta[[z,z^{-1}]].
\end{align}
In view of Lemma \ref{lem:ideal}, $\g^\zeta$ is  a quotient algebra of $\bar{\g}^\zeta$ with the multiplication given by
\begin{equation}\label{untwisted-lb}\begin{split}
&[a^{\alpha,\zeta}(z),b^{\beta,\zeta}(w)]\\
=&\alpha^{\epsilon-1}\sum_{\gamma\in \Gamma}\sum_{i,j\geq 0}\beta^{(i+j)(\epsilon-1)}\(\pd{ w}{\zeta}^{j}(a_{(\alpha\beta^{-1},\gamma,i,j)}b)^{\gamma\beta,\zeta}(w)\)\Del{w,\zeta}{i}{z,w},
\end{split}\end{equation}
where $a,b\in A$ and $\alpha,\beta\in \Gamma$.

\begin{prpt}\label{prop:liealgebragzeta}
 $\g^\zeta$ is a Lie algebra under the multiplication \eqref{untwisted-lb}.
 \end{prpt}
 \begin{proof} We first prove the skew-symmetry.
For $a,b\in A$ and $\alpha,\beta\in\Gamma$,
we have
\begin{align*}
[a^{\alpha,\zeta}(z),b^{\beta,\zeta}(w)]
=&\alpha^{\epsilon-1}\sum_{\gamma\in \Gamma}\sum_{k,j\geq 0}\beta^{(k+j)(\epsilon-1)}\(\pd{ w}{\zeta}^{j}(a_{(\alpha\beta^{-1},\gamma,k,j)}b)^{\gamma\beta,\zeta}(w)\)\Del{w,\zeta}{k}{z,w}.
\end{align*}
On the other hand, we have
\begin{align*}
&-[b^{\beta,\zeta}(w),a^{\alpha,\zeta}(z)]\\
=&-\beta^{\epsilon-1}\sum_{\gamma\in \Gamma}\sum_{i,j\geq 0}\alpha^{(i+j)(\epsilon-1)}\(\pd{z}{\zeta}^{j}(b_{(\beta\alpha^{-1},\gamma,i,j)}a)^{\gamma\alpha,\zeta}(z)\)\Del{z,\zeta}{i}{w,z}\\
=&\,\beta^{\epsilon-1}\sum_{\gamma\in \Gamma}\sum_{i,j\geq 0}\frac{(-1)^{i+1}}{i!}\alpha^{(i+j)(\epsilon-1)}\pd{w}{\zeta}^{i}\(\(\pd{w}{\zeta}^{j}(b_{(\beta\alpha^{-1},\gamma,i,j)}a)^{\gamma\alpha,\zeta}(w)\)\Del{w,\zeta}{0}{z,w}\)\\
=&\,\beta^{\epsilon-1}\sum_{\gamma\in \Gamma}\sum_{i,j\geq 0}\sum_{k=0}^i\frac{(-1)^{i+1}}{(i-k)!}\alpha^{(i+j)(\epsilon-1)}\(\pd{w}{\zeta}^{i+j-k}(b_{(\beta\alpha^{-1},\gamma,i,j)}a)^{\gamma\alpha,\zeta}(w)\)\Del{w,\zeta}{k}{z,w}\\
=&\,\beta^{\epsilon-1}\sum_{\gamma\in \Gamma}\sum_{j,k\geq 0}\sum_{i=0}^j\frac{(-1)^{i+k+1}}{i!}\alpha^{(j+k)(\epsilon-1)}\(\pd{w}{\zeta}^{j}(b_{(\beta\alpha^{-1},\gamma,i+k,j-i)}a)^{\gamma\alpha,\zeta}(w)\)\Del{w,\zeta}{k}{z,w}.
\end{align*}

Thus, the proof of the skew-symmetry can be reduced to the proof of the following relations $(k\in \N)$:
\begin{equation}\label{2-1}\begin{split}
&\beta^{k(\epsilon-1)}\alpha^{\epsilon-1}\sum_{\gamma\in \Gamma}\sum_{j\geq 0}\beta^{j(\epsilon-1)}\pd{ w}{\zeta}^{j}(a_{(\alpha\beta^{-1},\gamma,k,j)}b)^{\gamma\beta,\zeta}(w)\\
=&\,\alpha^{k(\epsilon-1)}\beta^{\epsilon-1}\sum_{\gamma\in \Gamma}\sum_{j\geq 0}\sum_{i=0}^j\frac{(-1)^{i+k+1}}{i!}\alpha^{j(\epsilon-1)}\pd{w}{\zeta}^{j}(b_{(\beta\alpha^{-1},\gamma,i+k,j-i)}a)^{\gamma\alpha,\zeta}(w).
\end{split}\end{equation}
By taking $\lambda=\beta\alpha^{-1}$ and replacing $w$ with $\beta w$ in \eqref{skew-sym}, we have
\begin{equation*}\begin{split}
&\sum_{\gamma\in \Gamma}\sum_{j\geq0}\beta^{j(\epsilon-1)}\(w^\epsilon\frac{\partial}{\partial w}\)^{j}(a_{(\alpha\beta^{-1},\gamma,k,j)}b)(\gamma\beta w)\\
=&\,(\beta\alpha^{-1})^{(-k+1)(\epsilon-1)}\sum_{\gamma\in \Gamma}\sum_{j\geq0}\sum_{i=0}^j\frac{(-1)^{i+k+1}}{i!}\alpha^{j(\epsilon-1)}\(w^\epsilon\frac{\partial}{\partial w}\)^{j}(b_{(\beta\alpha^{-1},\gamma,i+k,j-i)}a)(\gamma\alpha w).
\end{split}\end{equation*}
This  implies  \eqref{2-1} and hence the skew-symmetry.

Next, we show that the Jacobi identity holds in $\g^\zeta$. For  $a,b,c\in A$ and $\alpha,\beta,\lambda\in\Gamma$,
we have
\begin{eqnarray*}
&&\quad[a^{\alpha,\zeta}(z_1),[b^{\beta,\zeta}(z_2),c^{\lambda,\zeta}(z_3)]]\\
&&=\beta^{\epsilon-1}\sum_{\xi\in \Gamma}\sum_{i,j\geq 0}\lambda^{(i+j)(\epsilon-1)}\(\pd{z_3}{\zeta}^{j}\big[a^{\alpha,\zeta}(z_1),(b_{(\beta\lambda^{-1},\xi,i,j)}c)^{\xi\lambda,\zeta}(z_3)\big]\)
\Del{z_3,\zeta}{i}{z_2,z_3}\\
&&=(\alpha\beta)^{\epsilon-1}\sum_{\substack{\gamma,\xi\in \Gamma\\i,j,k,l\geq0}}\lambda^{(i+j)(\epsilon-1)}(\lambda\xi)^{(k+l)(\epsilon-1)}\Del{z_3,\zeta}{i}{z_2,z_3}\\
&&\quad\pd{z_3}{\zeta}^{j}\(\(\pd{z_3}{\zeta}^{l}
\(a_{(\alpha(\xi\lambda)^{-1},\gamma,k,l)}(b_{(\beta\lambda^{-1},\xi,i,j)}c)\)^{\gamma\xi\lambda,\zeta}(z_3)\)\Del{z_3,\zeta}{k}{z_1,z_3}\)\\
&&=(\alpha\beta)^{\epsilon-1}\sum_{\substack{\gamma,\xi\in \Gamma\\i,j,k,l\geq0}}\sum_{s=0}^j\binom{j}{s}\frac{(k+s)!}{k!}\lambda^{(i+j)(\epsilon-1)}(\lambda\xi)^{(k+l)(\epsilon-1)}\cdot\\
&&\quad\(\pd{z_3}{\zeta}^{l+j-s}
\(a_{(\alpha(\xi\lambda)^{-1},\gamma,k,l)}(b_{(\beta\lambda^{-1},\xi,i,j)}c)\)^{\gamma\xi\lambda,\zeta}(z_3)\)
\Del{z_3,\zeta}{k+s}{z_1,z_3}\Del{z_3,\zeta}{i}{z_2,z_3}\\
&&=(\alpha\beta)^{\epsilon-1}\sum_{\substack{\gamma,\xi\in \Gamma\\i,j,k,l,s\geq0}}\binom{j+s}{s}\frac{(k+s)!}{k!}\lambda^{(i+j+s)(\epsilon-1)}
(\lambda\xi)^{(k+l)(\epsilon-1)}\cdot\\
&&\quad\(\pd{z_3}{\zeta}^{l+j}\(a_{(\alpha(\xi\lambda)^{-1},\gamma,k,l)}(b_{(\beta\lambda^{-1},\xi,i,j+s)}c)\)^{\gamma\xi\lambda,\zeta}(z_3)\)
\Del{z_3,\zeta}{k+s}{z_1,z_3}\Del{z_3,\zeta}{i}{z_2,z_3}\\
&&=(\alpha\beta)^{\epsilon-1}\sum_{\substack{\gamma,\xi\in \Gamma\\i,k,l\geq0}}\sum_{j=0}^l\sum_{s=0}^k\binom{j+s}{s}\frac{k!}{(k-s)!}\lambda^{(i+k+l)(\epsilon-1)}\xi^{(k+l-j-s)(\epsilon-1)}\cdot\\
&&\quad\(\pd{z_3}{\zeta}^{l}\(a_{(\alpha(\xi\lambda)^{-1},\gamma,k-s,l-j)}(b_{(\beta\lambda^{-1},\xi,i,j+s)}c)\)^{\gamma\xi\lambda,\zeta}(z_3)\)
\Del{z_3,\zeta}{k}{z_1,z_3}\Del{z_3,\zeta}{i}{z_2,z_3}.
\end{eqnarray*}

Swapping  $a^{\alpha,\zeta}(z_1)$  with $b^{\beta,\zeta}(z_2)$ in the above equation yields the following equation:
\begin{eqnarray*}
&&\quad[b^{\beta,\zeta}(z_2),[a^{\alpha,\zeta}(z_1),c^{\lambda,\zeta}(z_3)]]\\
&&=(\alpha\beta)^{\epsilon-1}\sum_{\substack{\gamma,\xi\in \Gamma\\i,k,l\geq0}}\sum_{j=0}^l\sum_{s=0}^k\binom{j+s}{s}\frac{k!}{(k-s)!}\lambda^{(i+k+l)(\epsilon-1)}\xi^{(k+l-j-s)(\epsilon-1)}\cdot\\
&&\quad\(\pd{z_3}{\zeta}^{l}\(b_{(\beta(\xi\lambda)^{-1},\gamma,k-s,l-j)}(a_{(\alpha\lambda^{-1},\xi,i,j+s)}c)\)^{\gamma\xi\lambda,\zeta}(z_3)\)
\Del{z_3,\zeta}{k}{z_2,z_3}\Del{z_3,\zeta}{i}{z_1,z_3}.
\end{eqnarray*}

On the other hand, we have
\begin{eqnarray*}
&&\quad[[a^{\alpha,\zeta}(z_1),b^{\beta,\zeta}(z_2)],c^{\lambda,\zeta}(z_3)]\\
&&=\alpha^{\epsilon-1}\sum_{\xi\in \Gamma}\sum_{i,j\geq 0}\beta^{(i+j)(\epsilon-1)}\(\pd{z_2}{\zeta}^{j}
\big[(a_{(\alpha\beta^{-1},\xi,i,j)}b)^{\xi\beta,\zeta}(z_2),c^{\lambda,\zeta}(z_3)\big]\)
\Del{z_2,\zeta}{i}{z_1,z_2}\\
&&=\sum_{\substack{\xi,\gamma\in \Gamma\\i,j,k,l\geq 0}}(-1)^{i+j}\frac{(k+j)!}{k!}(\xi\beta\alpha)^{\epsilon-1}\beta^{(i+j)(\epsilon-1)}\lambda^{(k+l)(\epsilon-1)}\cdot\\
&&\quad\(\pd{z_3}{\zeta}^{l}\((a_{(\alpha\beta^{-1},\xi,i,j)}b)_{(\xi\beta\lambda^{-1},\gamma,k,l)}c\)^{\gamma\lambda,\zeta}(z_3)\)
\Del{z_3,\zeta}{k+j}{z_2,z_3}\Del{z_1,\zeta}{i}{z_1,z_2}\\
&&=\sum_{\substack{\xi,\gamma\in \Gamma\\i,k,l\geq 0}}\sum_{j=0}^k(-1)^{i+j}\frac{1}{i!(k-j)!}(\xi\beta\alpha)^{\epsilon-1}\beta^{(i+j)(\epsilon-1)}\lambda^{(k+l-j)(\epsilon-1)}\cdot\\
&&\quad\(\pd{z_3}{\zeta}^{l}\((a_{(\alpha\beta^{-1},\xi,i,j)}b)_{(\xi\beta\lambda^{-1},\gamma,k-j,l)}c\)^{\gamma\lambda,\zeta}(z_3)\)\\
&&\quad\pd{z_3}{\zeta}^{k}\pd{z_1}{\zeta}^{i}\(\Del{z_3,\zeta}{0}{z_2,z_3}\Del{z_3,\zeta}{0}{z_1,z_3}\)\\
&&=\sum_{\substack{\xi,\gamma\in \Gamma\\i,k,l\geq 0}}\sum_{j=0}^k\sum_{s=0}^{k}\binom{i+s}{s}\frac{(-1)^{j}k!}{(k-j)!}(\xi\beta\alpha)^{\epsilon-1}\beta^{(i+j)(\epsilon-1)}\lambda^{(k+l-j)(\epsilon-1)}\cdot\\
&&\quad\(\pd{z_3}{\zeta}^{l}\((a_{(\alpha\beta^{-1},\xi,i,j)}b)_{(\xi\beta\lambda^{-1},\gamma,k-j,l)}c\)^{\gamma\lambda,\zeta}(z_3)\)
\Del{z_3,\zeta}{k-s}{z_2,z_3}\Del{z_3,\zeta}{i+s}{z_1,z_3}\\
&&=\sum_{\substack{\xi,\gamma\in \Gamma\\i,k,l,s\geq 0}}\sum_{j=0}^{k+s}\binom{i+s}{s}\frac{(-1)^{j}(k+s)!}{(k+s-j)!}(\xi\beta\alpha)^{\epsilon-1}\beta^{(i+j)(\epsilon-1)}\lambda^{(k+s+l-j)(\epsilon-1)}\cdot\\
&&\quad\(\pd{z_3}{\zeta}^{l}\((a_{(\alpha\beta^{-1},\xi,i,j)}b)_{(\xi\beta\lambda^{-1},\gamma,k+s-j,l)}c\)^{\gamma\lambda,\zeta}(z_3)\)
\Del{z_3,\zeta}{k}{z_2,z_3}\Del{z_3,\zeta}{i+s}{z_1,z_3}\\
&&=\sum_{\substack{\xi,\gamma\in \Gamma\\i,k,l\geq 0}}\sum_{s=0}^i\sum_{j=0}^{k+s}\binom{i}{s}\frac{(-1)^{j}(k+s)!}{(k+s-j)!}(\xi\beta\alpha)^{\epsilon-1}\beta^{(i+j-s)(\epsilon-1)}\lambda^{(k+s+l-j)(\epsilon-1)}\cdot\\
&&\quad\(\pd{z_3}{\zeta}^{l}\((a_{(\alpha\beta^{-1},\xi,i-s,j)}b)_{(\xi\beta\lambda^{-1},\gamma,k+s-j,l)}c\)^{\gamma\lambda,\zeta}(z_3)\)
\Del{z_3,\zeta}{k}{z_2,z_3}\Del{z_3,\zeta}{i}{z_1,z_3}.
\end{eqnarray*}

Thus we have shown that the Jacobi identity $$[a^{\alpha,\zeta}(z_1),[b^{\beta,\zeta}(z_2),c^{\lambda,\zeta}(z_3)]]=[[a^{\alpha,\zeta}(z_1),b^{\beta,\zeta}(z_2)],c^{\lambda,\zeta}(z_3)]+[b^{\beta,\zeta}(z_2),[a^{\alpha,\zeta}(z_1),c^{\lambda,\zeta}(z_3)]]$$
 is equivalent to the identity ($i,k\in\N$):
 \begin{eqnarray}\label{4-2}
\nno&&\quad(\alpha\beta)^{\epsilon-1}\sum_{\substack{\gamma,\xi\in \Gamma\\l\geq0}}\sum_{j=0}^l\sum_{s=0}^i\binom{j+s}{s}\frac{i!}{(i-s)!}\lambda^{(i+k+l)(\epsilon-1)}\xi^{(i+l-j-s)(\epsilon-1)}\cdot\\
\nno&&\quad\pd{z_3}{\zeta}^{l}\(a_{(\alpha(\xi\lambda)^{-1},\gamma,i-s,l-j)}(b_{(\beta\lambda^{-1},\xi,k,j+s)}c)\)^{\gamma\xi\lambda,\zeta}(z_3)\\
\nno&&=(\beta\alpha)^{\epsilon-1}\sum_{\substack{\gamma,\xi\in \Gamma\\l\geq 0}}\sum_{s=0}^i\sum_{j=0}^{k+s}\binom{i}{s}\frac{(-1)^{j}(k+s)!}{(k+s-j)!}\xi^{\epsilon-1}\beta^{(i+j-s)(\epsilon-1)}\lambda^{(k+s+l-j)(\epsilon-1)}\cdot\\
&&\quad\pd{z_3}{\zeta}^{l}\((a_{(\alpha\beta^{-1},\xi,i-s,j)}b)_{(\xi\beta\lambda^{-1},\gamma,k+s-j,l)}c\)^{\gamma\lambda,\zeta}(z_3)\\
\nno&&\quad+(\alpha\beta)^{\epsilon-1}\sum_{\substack{\gamma,\xi\in \Gamma\\l\geq0}}\sum_{j=0}^l\sum_{s=0}^k\binom{j+s}{s}\frac{k!}{(k-s)!}\lambda^{(i+k+l)(\epsilon-1)}\xi^{(k+l-j-s)(\epsilon-1)}\cdot\\
\nno&&\quad\pd{z_3}{\zeta}^{l}\(b_{(\beta(\xi\lambda)^{-1},\gamma,k-s,l-j)}(a_{(\alpha\lambda^{-1},\xi,i,j+s)}c)\)^{\gamma\xi\lambda,\zeta}(z_3).
\end{eqnarray}
Note that the identity \eqref{4-2} can be simplified as follows:
\begin{eqnarray*}
&&\quad\sum_{\substack{\gamma,\xi\in \Gamma\\l\geq0}}\sum_{j=0}^l\sum_{s=0}^i\binom{j+s}{s}\frac{i!}{(i-s)!}\lambda^{l(\epsilon-1)}\xi^{(i+l-j-s)(\epsilon-1)}\cdot\\
&&\quad\pd{z_3}{\zeta}^{l}\(a_{(\alpha(\xi\lambda)^{-1},\gamma,i-s,l-j)}(b_{(\beta\lambda^{-1},\xi,k,j+s)}c)\)^{\gamma\xi\lambda,\zeta}(z_3)\\
&&=\sum_{\substack{\gamma,\xi\in \Gamma\\l\geq 0}}\sum_{s=0}^i\sum_{j=0}^{k+s}\binom{i}{s}\frac{(-1)^{j}(k+s)!}{(k+s-j)!}\xi^{\epsilon-1}(\beta\lambda^{-1})^{(i+j-s)(\epsilon-1)}\lambda^{l(\epsilon-1)}\cdot\\
&&\quad\pd{z_3}{\zeta}^{l}\((a_{(\alpha\beta^{-1},\xi,i-s,j)}b)_{(\xi\beta\lambda^{-1},\gamma,k+s-j,l)}c\)^{\gamma\lambda,\zeta}(z_3)\\
&&\quad+\sum_{\substack{\gamma,\xi\in \Gamma\\l\geq0}}\sum_{j=0}^l\sum_{s=0}^k\binom{j+s}{s}\frac{k!}{(k-s)!}\lambda^{l(\epsilon-1)}\xi^{(k+l-j-s)(\epsilon-1)}\cdot\\
&&\quad\pd{z_3}{\zeta}^{l}\(b_{(\beta(\xi\lambda)^{-1},\gamma,k-s,l-j)}(a_{(\alpha\lambda^{-1},\xi,i,j+s)}c)\)^{\gamma\xi\lambda,\zeta}(z_3),
\end{eqnarray*}
which  follows from \eqref{jacobi-quasi} (replacing $\lambda$ with $\beta\lambda^{-1}$,  taking $\eta=\alpha\lambda^{-1}$ and  replacing $z$ with $\lambda z$ therein).
This completes the proof of the proposition.
\end{proof}

Note that Lemma \ref{lem:ideal} and Proposition \ref{prop:liealgebragzeta} imply the Theorem \ref{thm:main1}\,(I).

\subsection{Proof of Theorem \ref{thm:main1} (II)} In this subsection we prove the second part of Theorem \ref{thm:main1}.
For $\lambda\in \Gamma$, we  define a linear transformation $\bar{\sigma}_\lambda$ on $\bar{\g}^\epsilon$ by
\begin{align*}
\bar{\sigma}_\lambda\(\bar{a}^{\alpha,\epsilon}(m)\)=\lambda^{-m+\epsilon-1}\bar{a}^{\alpha\lambda^{-1},\epsilon}(m),
\end{align*}
where $a\in A,\ \alpha\in \Gamma$ and $m\in \Z$.  In term of the generating functions, this is equivalent to
\[\bar{\sigma}_\lambda(\bar{a}^{\alpha,\epsilon}(z))=\bar{a}^{\alpha\lambda^{-1},\epsilon}(\lambda z).\]

\begin{lemt}\label{lem:barsigma} For any $\lambda\in \Gamma$, one has  $\bar{\sigma}_\lambda(\bar{\g}^\epsilon_0)= \bar{\g}^\epsilon_0$.
\end{lemt}
\begin{proof} Recall that  $\bar{\g}^\epsilon_0$ is spanned by the
coefficients of the generating functions in $\ker \bar{\psi}^\epsilon$.
We prove the lemma by showing that for any $u(z)\in\ker\bar{\psi}^\epsilon$ and $\lambda\in \Gamma$, there exists a
$v(z)\in\ker\bar{\psi}^\epsilon $ such that  $\bar\sigma_\lambda(u(z))=v(\lambda z)$.

Assume that
\[u(z)=\sum_{i=1}^k\mu_i\pd{z}{\epsilon}^{n_i}\overline{a_i}^{\alpha_i,\epsilon}(z),\]
 where $\mu_i\in\C,n_i\in\N, a_i\in A$ and $\alpha_i\in\Gamma$.
We define
\[v(z)=\sum_{i=1}^k\mu_i\lambda^{n_i(1-\epsilon)}\pd{z}{\epsilon}^{n_i}\overline{a_i}^{\alpha_i\lambda^{-1},\epsilon}(z).\]
By definition we have
\[\bar\sigma_\lambda(u(z))=\sum_{i=1}^k\mu_i \pd{z}{\epsilon}^{n_i}\overline{a_i}^{\alpha_i\lambda^{-1},\epsilon}(\lambda z)=v(\lambda z).\]
Since
\[\bar{\psi}^\epsilon(u(z))=\sum_{i=1}^k\mu_i\pd{z}{\epsilon}^{n_i}a_i(\alpha_iz)=0,\] we obtain
\[\bar{\psi}^\epsilon(v(z))=\sum_{i=1}^k\mu_i\lambda^{n_i(1-\epsilon)}\pd{z}{\epsilon}^{n_i}a_i(\alpha_i\lambda^{-1}z) =\bigg(\sum_{i=1}^k\mu_i\pd{w}{\epsilon}^{n_i}a_i(\alpha_iw)\bigg)\Big|_{w=\lambda ^{-1}z}=0.\]
This shows that $v(z)\in\ker\bar{\psi}^\epsilon$, as desired.
\end{proof}

 Lemma \ref{lem:barsigma} implies that $\bar{\sigma}_\lambda$ induces a linear transformation, say $\sigma_\lambda$, on $\g^\epsilon$ such that
\[\sigma_\lambda\(a^{\alpha,\epsilon}(m)\)=\lambda^{-m+\epsilon-1}a^{\alpha\lambda^{-1},\epsilon}(m),\]
where $a\in A,\ \alpha\in \Gamma$ and $m\in \Z,$ or equivalently, such that
\[\sigma_\lambda(a^{\alpha,\epsilon}(z))=a^{\alpha\lambda^{-1},\epsilon}(\lambda z).\]

\begin{lemt}\label{lem:sigma} For every $\lambda\in \Gamma$, $\sigma_\lambda$ is a Lie  automorphism of the Lie algebra $\g^\epsilon$.  Furthermore,
for any $u,v\in \g^\epsilon$, $[\sigma_\lambda(u),v]=0$ for all but finitely many $\lambda\in \Gamma$.
\end{lemt}

\begin{proof}It is straightforward to see that $\sigma_\lambda$ is bijective. For $a,b\in A$ and $\alpha,\beta\in\Gamma$,  from \eqref{untwisted-lb} and
\eqref{eq:partial-delta0}
it follows that
\begin{equation*}\begin{split}
&\,[\sigma_\lambda( a^{\alpha,\epsilon}(z)),\sigma_\lambda(b^{\beta,\epsilon}(w))]=[ a^{\alpha\lambda^{-1},\epsilon}(\lambda z),b^{\beta\lambda^{-1},\epsilon}(\lambda w)]\\
=&\,(\alpha\lambda^{-1})^{\epsilon-1}\sum_{\gamma\in \Gamma}\sum_{i,j\geq 0}\beta^{(i+j)(\epsilon-1)}\(\pd{ w}{\epsilon}^{j}(a_{(\alpha\beta^{-1},\gamma,i,j)}b)^{\gamma\beta\lambda^{-1},\epsilon}(\lambda  w)\)\Del{w,\epsilon}{i}{\lambda z,\lambda w}\\
=&\,\alpha^{\epsilon-1}\sum_{\gamma\in \Gamma}\sum_{i,j\geq 0}\beta^{(i+j)(\epsilon-1)}\(\pd{ w}{\epsilon}^{j}(a_{(\alpha\beta^{-1},\gamma,i,j)}b)^{\gamma\beta\lambda^{-1},\epsilon}(\lambda  w)\)\Del{w,\epsilon}{i}{ z, w}\\
=&\,\sigma_\lambda\big([a^{\alpha,\epsilon}(z),b^{\beta,\epsilon}(w)]\big).
\end{split}\end{equation*}
This proves the first assertion. For the second one, let $a,b\in A$ and $\alpha,\beta\in \Gamma$ be fixed. For every $\lambda\in \Gamma$, we have
\begin{equation*}\begin{split}
&\,[\sigma_\lambda( a^{\alpha,\epsilon}(z)),b^{\beta,\epsilon}(w)]=[ a^{\alpha\lambda^{-1},\epsilon}(\lambda z),b^{\beta,\epsilon}(w)]\\
=&\,(\alpha\lambda^{-1})^{\epsilon-1}\sum_{\gamma\in \Gamma}\sum_{i,j\geq 0}\beta^{(i+j)(\epsilon-1)}\(\pd{ w}{\epsilon}^{j}(a_{(\alpha\lambda^{-1}\beta^{-1},\gamma,i,j)}b)^{\gamma\beta,\epsilon}(w)\)\Del{w,\epsilon}{i}{\lambda z,w}.
\end{split}\end{equation*}
Note that the Lie relation \eqref{[a,b]} implies that $(a_{(\alpha\lambda^{-1}\beta^{-1},\gamma,i,j)}b)^{\gamma\beta,\epsilon}(w)=0$ for all but finitely many $\lambda,\gamma\in\Gamma,i,j\in\N$. Then the second assertion follows immediately.
\end{proof}

Recall the multiplication $[\cdot,\cdot]_\Gamma$ on $\g^\epsilon$
 and the subspace $\g^\epsilon_\Gamma$ of $\g^\epsilon$ defined in Theorem \ref{thm:main1} (II).

\begin{prpt}\label{lem:thm1-2}  $\g^\epsilon_\Gamma$
is a two-sided ideal of the nonassociative algebra $\g^\epsilon$ under the multiplication $[\cdot,\cdot]_\Gamma$, and the quotient algebra
$\g^\epsilon[\Gamma]=\g^\epsilon/\g^\epsilon_\Gamma$ is a Lie algebra.
\end{prpt}
\begin{proof}
By definition we have $[u,v]_\Gamma=\sum_{\lambda\in \Gamma}[\sigma_\lambda(u),v]$  for $u,v\in \g^\epsilon$, and note that the subspace
\[\g^\epsilon_\Gamma=\te{Span}\{\sigma_\lambda(u)-u\mid \lambda\in \Gamma, u\in \g^\epsilon\}.\]
Then the assertion follows from  Lemma \ref{lem:sigma} and \cite[Lemma 4.1]{Li3}.
\end{proof}

For $a\in A, \alpha\in \Gamma$ and $m\in \Z$, set
$a^{\alpha}(m)=a^{\alpha,\epsilon}(m)+\g^\epsilon_\Gamma\in \g^\epsilon[\Gamma]$ and
\begin{align}\label{a^alpha(z)}
a^{\alpha}(z)=\sum_{m\in \Z} a^{\alpha}(m)z^{-m+\epsilon-1}\in \g^\epsilon[\Gamma][[z,z^{-1}]].
\end{align}
The following result gives the commutators among these generating functions.

\begin{lemt}For $a,b\in A$ and $\alpha,\beta\in \Gamma$, one has
\begin{align}\label{universal-twisted-lb}\begin{split}
&\,[a^{\alpha}(z),b^{\beta}(w)]_\Gamma\\
=&\,\sum_{\lambda,\gamma\in \Gamma}\sum_{i,j\geq 0}\alpha^{\epsilon-1}\beta^{(i+j)(\epsilon-1)}\(\pd{ w}{\epsilon}^{j}(a_{(\lambda\alpha\beta^{-1},\gamma,i,j)}b)^{\gamma\beta}(w)\)\Del{w,\epsilon}{i}{ z,\lambda w}.
\end{split}\end{align}
\end{lemt}
\begin{proof}
By definition and \eqref{eq:partial-delta0}, we have
\begin{align*}
&\,[a^{\alpha}(z), b^{\beta}(w)]_{\Gamma}
=\sum_{\lambda\in \Gamma}[ a^{\alpha\lambda^{-1}}(\lambda z), b^{\beta}(w)]\\
=&\,\sum_{\lambda,\gamma\in \Gamma}\sum_{i,j\geq 0}(\lambda^{-1}\alpha)^{\epsilon-1}\beta^{(i+j)(\epsilon-1)}\(\pd{ w}{\epsilon}^{j}(a_{(\alpha\lambda^{-1}\beta^{-1},\gamma,i,j)}b)^{\gamma\beta}(w)\)\Del{w,\epsilon}{i}{\lambda z,w}\\
=&\,\sum_{\lambda,\gamma\in \Gamma}\sum_{i,j\geq 0}\alpha^{\epsilon-1}\beta^{(i+j)(\epsilon-1)}\(\pd{ w}{\epsilon}^{j}(a_{(\lambda\alpha\beta^{-1},\gamma,i,j)}b)^{\gamma\beta}(w)\)\Del{w,\epsilon}{i}{ z,\lambda w}.
\end{align*}
\end{proof}

Let $\bar{\varphi}_{\g}$ be the linear map from $\bar{\g}^\epsilon$ to $\g$ defined by
\begin{align*}
\bar{\varphi}_{\g}(\bar{a}^{\alpha,\epsilon}(m))=\alpha^{-m+\epsilon-1}a(m)
\end{align*}
for $a\in A, \alpha\in\Gamma$ and $m\in \Z$.
Note that we have
\[\bar{\varphi}_{\g}\(\(z^\epsilon\frac{\partial}{\partial z}\)^n \bar{a}^{\alpha,\epsilon}(z)\)
=\(z^\epsilon\frac{\partial}{\partial z}\)^n a(\alpha z)\] for $n\in \N, a\in A$ and $\alpha\in \Gamma$.
This implies that $\bar{\varphi}_{\g}(\ker \bar{\psi}^\epsilon)=0$ (see \eqref{intro:psizeta}) and hence $\bar{\varphi}_{\g}(\bar{\g}^\epsilon_0)=0$.
Thus, $\bar{\varphi}_{\g}$ induces
a linear map, say $\varphi_{\g}$, from $\g^\epsilon=\bar{\g}^\epsilon/\bar{\g}^\epsilon_0$ to $\g$ such that
\begin{align*}
\varphi_{\g}(a^{\alpha,\epsilon}(m))=\alpha^{-m+\epsilon-1}a(m)
\end{align*}
for $a\in A, \alpha\in\Gamma$ and $m\in \Z$.
Furthermore, for $a\in A, \lambda,\alpha\in \Gamma$ and $m\in \Z$, we have
\begin{align*}
\varphi_{\g}(\sigma_\lambda(a^{\alpha,\epsilon}(m))-a^{\alpha,\epsilon}(m))
=\varphi_{\g}(\lambda^{-m+\epsilon-1}a^{\alpha\lambda^{-1},\epsilon}(m)
-a^{\alpha,\epsilon}(m))=0,
\end{align*}
which shows that  $\varphi_{\g}$ vanishes on $\g^\epsilon_\Gamma$.
Then $\varphi_{\g}$ induces a linear map, say $\varphi_{\g,\Gamma}$, from $\g^\epsilon[\Gamma]=\g^\epsilon/\g^\epsilon_\Gamma$ to $\g$
such that
\begin{align*}
\varphi_{\g,\Gamma}(a^{\alpha}(m))=\alpha^{-m+\epsilon-1}a(m)
\end{align*} for $a\in A, \alpha\in\Gamma$ and $m\in \Z$.
We remark that the map $\varphi_{\g,\Gamma}$ is just the linear map introduced in Theorem \ref{thm:main1}\,(II).

\begin{prpt}\label{lem:varphi-surjecitve}
The map $\varphi_{\g,\Gamma}$ is a surjective Lie homomorphism.
\end{prpt}
\begin{proof}
The map $\varphi_{\g,\Gamma}$ is surjective as $\g$ is  spanned by the elements $a(m)$ for  $a\in A$ and $m\in \Z$.
Note that for $a\in A$ and $\alpha\in \Gamma$, we have
$a^\alpha(z)=\sigma_\alpha(a^\alpha(z))=a^1(\alpha z)$. This gives that
\begin{align}\label{eq:spangeg}
 \g^\epsilon[\Gamma]=\text{Span}\{a^1(m)\mid a\in A, m\in \Z\}.
 \end{align}
Furthermore, from \eqref{universal-twisted-lb} we have
\begin{align*}
\varphi_{\g,\Gamma}([a^{1}(z),b^{1}(w)]_\Gamma)
=&\sum_{\lambda,\gamma\in \Gamma}\sum_{i,j\geq 0}\(\pd{ w}{\epsilon}^{j}(a_{(\lambda,\gamma,i,j)}b)(\gamma w)\)\Del{w,\epsilon}{i}{ z,\lambda w}\\
=&[a(z),b(w)]=[\varphi_{\g,\Gamma}(a^{1}(z)),\varphi_{\g,\Gamma}(b^{1}(w))]
\end{align*}
for any  $a,b\in A$. This proves that $\varphi_{\g,\Gamma}$ is a surjective Lie homomorphism.
\end{proof}

To be continuous, here we explain the definition of maximality of $\g$ given in the introduction.
Let $\tilde{\g}$ be the complex vector space with a basis
\[\{\tilde{a}(m)\mid a\in A, m\in \Z\}.\]
For $a\in A$, set $\tilde{a}(z)=\sum_{m\in \Z} \tilde{a}(m)z^{-m+\epsilon-1}$.
Let $\tilde{\mathcal{A}}$ be the subspace of $\tilde{\g}[[z,z^{-1}]]$ spanned by the (linearly independent) elements
$\(z^\epsilon\frac{\partial}{\partial z}\)^n \tilde{a}(\alpha z)$ for $n\in \N, a\in A$ and $\alpha\in \Gamma$, and let
 $\tilde{\varphi}:\tilde{\mathcal{A}}\rightarrow \g[[z,z^{-1}]]$ be the linear map defined by
\[\tilde{\varphi}\(\(z^\epsilon\frac{\partial}{\partial z}\)^n \tilde{a}(\alpha z)\)
=\(z^\epsilon\frac{\partial}{\partial z}\)^n a(\alpha z)\]
for $n\in \N, a\in A$ and $\alpha\in \Gamma$.
Define $\tilde{\g}_0$ to be the subspace of $\tilde{\g}$ spanned by the coefficients of generating functions in $\ker \tilde{\varphi}$.
Then it is clear that $\g$ is maximal if and only if
the canonical surjective map
\begin{align}\label{eq:canmaxi} \tilde{\g}/\tilde{\g}_0\rightarrow \g,\quad \tilde{a}(m)+\tilde{\g}_0\mapsto a(m)\quad (a\in A,m\in\Z)\end{align}
is an isomorphism. We also give a criterion for the maximality of $\g$ as follows.

\begin{remt}\label{maximal-equi} Let $\CR$ be a subset of $\ker \tilde{\varphi}$, and let $\tilde{\g}_{\CR}$ be the subspace of $\tilde{\g}$ spanned
by the coefficients of the generating functions in $\CR$.
Consider the canonical surjective  map
\begin{align} \label{eq:canmaxiR}\tilde{\g}/\tilde{\g}_{\CR}\rightarrow \g,\quad \tilde{a}(m)+\tilde{\g}_{\CR}\mapsto a(m)\quad (a\in A,m\in\Z),\end{align}
 which is the composition of the canonical surjective map $\tilde{\g}/\tilde{\g}_{\CR}\rightarrow \tilde{\g}/\tilde{\g}_0$ and
\eqref{eq:canmaxi}.
Thus, if \eqref{eq:canmaxiR} is an isomorphism,  then  \eqref{eq:canmaxi} is also an isomorphism and hence $\g$ is maximal.
 \end{remt}

The following result together with Proposition \ref{lem:varphi-surjecitve} give a necessary and sufficient condition for the homomorphism $\varphi_{\g,\Gamma}$ to be an isomorphism.

\begin{prpt}\label{lem:maximal-eq-injective}
The homomorphism $\varphi_{\g,\Gamma}$ is injective if and only if $\g$ is maximal.
\end{prpt}
\begin{proof}
By \eqref{eq:spangeg}, we have a surjective linear map
\begin{align*}
\tilde\chi:\tilde\g\rightarrow \g^\epsilon[\Gamma],
\quad \tilde{a}(m)\mapsto a^1(m)\quad (a\in A, m\in \Z).
\end{align*}
Note that for $a\in A$ and $\alpha\in \Gamma$, one has
\begin{align*}
\tilde\chi(\tilde{a}(\alpha z))=a^1(\alpha z)= a^{\alpha}(z).
\end{align*}
This together with  \eqref{intro:psizeta} gives that $\tilde{\chi}(\ker\tilde\varphi)=0$.
Thus $\tilde\chi$ induces a linear map
\begin{align*}
\chi:\tilde\g/\tilde\g_0\rightarrow \g^\epsilon[\Gamma],\quad \tilde{a}(m)+\tilde\g_0\mapsto a^1(m).
\end{align*}

We claim that $\chi$ is a linear isomorphism. In fact,
 we consider the linear map
\begin{align*}
\bar\eta: \bar\g^\epsilon\rightarrow \tilde\g/\tilde\g_0,\quad \bar{a}^{\alpha,\epsilon}(m)\mapsto \alpha^{-m+\epsilon-1}\tilde{a}(m)+\tilde\g_0
\quad (a\in A, \alpha\in\Gamma, m\in \Z).
\end{align*}
Equivalently,  we have $\bar\eta(\bar{a}^{\alpha,\epsilon}(z))=\tilde{a}(\alpha z)$ for $a\in A$ and $\alpha\in \Gamma$.
This together with the definition of $\ker\tilde\varphi$ gives that  $\bar\eta(\ker \bar\psi^\epsilon)=0$ (see  \eqref{intro:psizeta}).
Thus $\bar\eta$ induces  a linear map
\begin{align*}
\eta: \g^\epsilon\rightarrow \tilde\g/\tilde\g_0,\quad a^{\alpha,\epsilon}(m)\mapsto \alpha^{-m+\epsilon-1}\tilde{a}(m)+\tilde\g_0.
\end{align*}
Note that for $a\in A$ and $\alpha,\lambda\in \Gamma$, we have
\begin{align*}
\eta(\sigma_\lambda(a^{\alpha,\epsilon}(z)))=\eta(a^{\alpha\lambda^{-1},\epsilon}(\lambda z))=\tilde{a}(\alpha\lambda^{-1}\lambda z)
=\tilde{a}(\alpha z)=\eta(a^{\alpha,\epsilon}(z)).
\end{align*}
Thus $\eta$ factors through $\g^\epsilon_\Gamma$ and  yields a  linear map
\begin{align*}
\eta_{\Gamma}:\g^\epsilon[\Gamma]\rightarrow \tilde\g/\tilde\g_0,\quad a^{\alpha}(m)\mapsto \alpha^{-m+\epsilon-1}\tilde{a}(m)+\tilde\g_0.
\end{align*}
It is clear that $\chi$ and $\eta_\Gamma$ are mutually invertible, which proves the claim.

Note that  the map \eqref{eq:canmaxi} is the composition of the maps $\chi$ and $\varphi_{\g,\Gamma}$.
Then the proposition follows from the Proposition  \ref{lem:varphi-surjecitve} and  $\chi$ is an isomorphism.
\end{proof}

 Propositions \ref{lem:thm1-2}, \ref{lem:varphi-surjecitve} and \ref{lem:maximal-eq-injective} give the Theorem \ref{thm:main1}\,(II).

\section{Proof of Theorem \ref{thm:main2}}

In this section, we present the proof of  Theorem \ref{thm:main2}.
As before, throughout this section, let $(\g,\CA,\epsilon)$ be a quasi vertex Lie algebra with $\Gamma$ the associated group.

\subsection{Proof of Theorem \ref{thm:main2} (I)}

We start with  the following notion introduced in \cite{DLM}.
\begin{dfnt}{\em  A {\em vertex Lie algebra} is a quadruple $(\CL,U,d,\rho)$ consisting of a Lie algebra $\CL$, a vector space $U$,
 a partially defined linear map $d$ from $U$ to $U$,
 and a linear map $\rho$ from $U\ot\C[t,t^{-1}]$ onto $\CL$ such that $\ker\rho=\te{Im} (d\ot 1+ 1\ot \frac{d}{dt})$, and that for any $a,b\in U$, there exist
   finitely many vectors $c_{i,j}, i,j=0,1,\dots,k$ in $U$ (depending on $a,b$),  such that
 \begin{equation}\label{eq:localvla}
[a(z),b(w)]=\sum_{i,j=0}^k \frac{1}{i!}\(\(\frac{\partial}{\partial w}\)^j c_{i,j}(w)\)\(\frac{\partial}{\partial w}\)^iz^{-1}\delta\(\frac{w}{z}\),
\end{equation}
 where $a(z)=\sum_{n\in\Z}\rho(a\ot t^n) z^{-n-1}$.}
\end{dfnt}
Let $(\CL,U,d,\rho)$ be a vertex Lie algebra.  Set
\begin{align}\label{eq:defL+}
\CL_-=\te{Span}\{\rho(a\ot t^n)\mid a\in U,n<0\}\quad\te{and}\quad \CL_+=\te{Span}\{\rho(a\ot t^n)\mid a\in U,n\geq0\}.\end{align}
Then both $\CL_-$ and $\CL_+$  are Lie subalgebras of $\CL$, and $\CL=\CL_-\oplus \CL_+$.
Write $\C$ for the one-dimensional trivial $\CL_+$-module, and form the induced $\CL$-module
\[V_{\CL}=\U(\CL)\ot_{\U(\CL_+)}\C.\]
Then $U$ can be identified as a subspace of $V_{\CL}$ through the  map $a\mapsto \rho(a\ot t^{-1})\vac$, where $\vac=1\ot 1$.
Furthermore, we have the following results from \cite{DLM}.

\begin{prpt}\label{prop:vavl} Let $(\CL,U,d,\rho)$ be a vertex Lie algebra.  Then there is a unique vertex algebra structure on $V_{\CL}$ with $\vac$ as the vacuum vector and with
$U$ as a generating set such that
$Y(a,z)=a(z)$ for $a\in U$.
\end{prpt}

Recall that $\g^0$ is the Lie algebra $\g^\zeta$ in Theorem \ref{thm:main1} for $\zeta=0$. We show that the Lie algebra $\g^0$ admits a vertex Lie algebra structure.
Let $\C[d]$ be the polynomial ring with the  variable $d$, and let $\widetilde{U}$ be a free $\C[d]$-module equipped with  a $\C[d]$-basis
\[\{\tilde{a}^{\alpha,0} \mid a\in A, \alpha\in \Gamma\}.\]
We view  $\g[[z,z^{-1}]]$ as a $\C[d]$-module such that
 \[d(u(z))=z^{\epsilon}\frac{\partial}{\partial z} u(z)\quad \text{for}\ u(z)\in \g[[z,z^{-1}]],\] and let
$\Psi:\widetilde{U}\rightarrow \g[[z,z^{-1}]]$ be the
 $\C[d]$-module homomorphism defined by
 \[\Psi(\tilde{a}^{\alpha,0})=a(\alpha z)\quad\text{for}\ a\in A, \alpha\in \Gamma.\]
Similarly, we view $\g^0[[z,z^{-1}]]$ as a $\C[d]$-module such that
 \[d( u(z))=\frac{\partial}{\partial z} u(z)\quad \text{for}\ u(z)\in \g^0[[z,z^{-1}]],\]  and let
$\Phi:\widetilde{U}\rightarrow \g^0[[z,z^{-1}]]$ be the
 $\C[d]$-module homomorphism defined by
 \[\Phi(\tilde{a}^{\alpha,0})=a^{\alpha,0}(z)\quad\text{for}\ a\in A, \alpha\in \Gamma.\]

\begin{lemt}\label{lemma:prerho} One has $\ker \Psi\subset \ker \Phi$.
\end{lemt}
\begin{proof} Let $u\in\ker \Psi$. We can write
\[u=\sum_{i=1}^k \mu_i d^{n_i}\widetilde{a_i}^{\alpha_i,0}\quad (\mu_i\in \C,
n_i\in \N, a_i\in A, \alpha_i\in \Gamma).\]
By \eqref{intro:psizeta} we have
\begin{align*}
\bar\psi^0\(\sum_{i=1}^k \mu_i \(\frac{\partial}{\partial z}\)^{n_i} \overline{a_i}^{\alpha_i,0}(z)\)= \sum_{i=1}^k \mu_i \(z^\epsilon\frac{\partial}{\partial z}\)^{n_i} a_i(\alpha_i z)=\Psi(u)=0.
\end{align*}
This implies that  $\Phi(u)=\sum_{i=1}^k \mu_i \(\frac{\partial}{\partial z}\)^{n_i} a_i^{\alpha_i,0}(z)$ $=0$, as required.
\end{proof}

Form the quotient $\C[d]$-module
\[U=\widetilde{U}/\ker \Psi.\]
Let $\tilde{\rho}$ be the linear map from $\widetilde{U}\ot \C[t,t^{-1}]$ to $\g^0$ defined by
\begin{align*}
\sum_{m\in \Z} \tilde{\rho}(u\ot t^m) z^{-m-1}=\Phi(u)\quad\text{for}\ u\in \widetilde{U}.
\end{align*}
In view of Lemma \ref{lemma:prerho}, we have $\ker \Psi\ot \C[t,t^{-1}]\subset \ker \tilde{\rho}$,  so $\tilde{\rho}$ induces a linear map
\begin{align*}
\rho:U\ot \C[t,t^{-1}]\cong (\widetilde{U}\ot \C[t,t^{-1}])/ (\ker\Psi\ot \C[t,t^{-1}])\rightarrow \g^0
\end{align*}
such that
\begin{align}\label{eq:defrho}
\rho \((d^n a^{\alpha,0})\ot t^m\)=\Res_z z^m\(\frac{\partial}{\partial z}\)^na^{\alpha,0}(z)=(-1)^nn! {m\choose n} a^{\alpha,0}(m-n)
\end{align}
for $n\in \N, a\in A,\alpha\in\Gamma$ and $ m\in\Z$, where $a^{\alpha,0}$ denotes the image of $\tilde{a}^{\alpha,0}$ in $U$.
In terms of the generating functions, \eqref{eq:defrho} is equivalent to
\begin{align}
\sum_{m\in\Z}\rho\((d^n a^{\alpha,0})\ot t^m\)z^{-m-1}=\(\frac{\partial}{\partial z}\)^na^{\alpha,0}(z).
\end{align}
With the above definitions,  we have the following result.

\begin{lemt}\label{lem:thm2-1}The quadruple $(\g^0,U,d,\rho)$ is a vertex Lie algebra.
\end{lemt}
\begin{proof} By \eqref{untwisted-lb}, the commutator $[a(z),b(w)]$ $(a,b\in U)$ has the desired
form as  \eqref{eq:localvla}.
Since  $\rho$ is surjective, it remains to prove that $\ker\rho=\te{Im}(d\ot 1+1\ot \frac{d}{dt})$.

Set $K=\te{Im}(d\ot 1+1\ot \frac{d}{dt})$. For $u\in \widetilde{U},m\in\Z$, we have
\begin{align*}
 &\tilde{\rho}\(d u\ot t^m\) =\Res_z z^m\Phi(du)=\Res_z z^m\frac{\partial}{\partial z}(\Phi(u))\\
 =&-\Res_z \(\frac{\partial}{\partial z}z^m\)\Phi(u)=-m\Res_z z^{m-1}\Phi(u)=-m\tilde{\rho}(u\ot t^{m-1}).
 \end{align*}
This implies that $K\subset\ker\rho$. Then $\rho$ induces a surjective  map
\[\rho_K:(U\ot\C[t,t^{-1}])/K\rightarrow \g^0,\quad a^{\alpha,0}\ot t^m+K\mapsto a^{\alpha,0}(m)\quad (a\in A,\alpha\in \Gamma,m\in \Z).\]

On the other hand, we consider the following linear map:
\[\bar\pi:\bar\g^0\rightarrow U\ot\C[t,t^{-1}]/K,\quad \bar{a}^{\alpha,0}(m)\mapsto a^{\alpha,0}\ot t^m+K\quad (a\in A,\alpha\in \Gamma,m\in \Z).\]
We claim that $\ker\bar\psi^0\subset \ker \bar\pi$ (see \eqref{intro:psizeta}).
Indeed, fix a vector
\[v(z)=\sum_{i=1}^k \mu_i \(\frac{\partial}{\partial z}\)^{n_i} \overline{a_i}^{\alpha_i,0}(z)\in \ker \bar\psi^0\quad (\mu_i\in \C,n_i\in \N,a_i\in A,\alpha_i\in\Gamma).\]
Then  we have
\[\Psi\(\sum_{i=1}^k \mu_i d^{n_i}\widetilde{a_i}^{\alpha_i,0}\)=\sum_{i=1}^k \mu_i \(z^\epsilon\frac{\partial}{\partial z}\)^{n_i}a_i(\alpha_i z)=\bar\psi^0(v(z))=0.\]
This implies that
\[\sum_{i=1}^k \mu_i d^{n_i}{a_i}^{\alpha_i,0}=0\quad\text{in}\quad U=\widetilde{U}/\ker \Psi.\]

  For $u\in U$, set
\begin{align*}u_K(z)=\sum_{m\in\Z}(u\ot t^m+K)z^{-m-1}\in\((U\ot\C[t,t^{-1}])/K\)[[z,z^{-1}]].
\end{align*}
Note that $\bar\pi(\bar a^{\alpha,0}(z))= (a^{\alpha,0})_K(z)$ and  $(d a^{\alpha,0})_K(z)=\frac{\partial}{\partial z}(a^{\alpha,0})_K(z)$  for $a\in A, \alpha\in \Gamma$, we obtain
\[\bar\pi(v(z))=\sum_{i=1}^k \mu_i \(\frac{\partial}{\partial z}\)^{n_i} (a_i^{\alpha_i,0})_K(z)= \big(\sum_{i=1}^k \mu_i d^{n_i}a_i^{\alpha_i,0}\big)_K(z)=0.\]
This  proves $\ker\bar\psi^0\subset \ker \bar\pi$. Then $\bar\pi$ induces a linear map $\pi$ from $\g^0$ to $U\ot\C[t,t^{-1}]/K$, and
the assertion $K=\ker\rho$ follows from the fact that
$\pi$ is
the inverse of $\rho_K$.
\end{proof}

Now we have shown that $(\g^0,U,d,\rho)$ is a vertex Lie algebra. Take $\mathcal{L}=\g^0$ in \eqref{eq:defL+} we
obtain a subalgebra $\g^0_+$ of $\g^0$. From  \eqref{eq:defrho}, we see that this subalgebra coincides with that defined in  \eqref{intro:g0+}.
Thus, it follows from Proposition \ref{prop:vavl} that  there is a
 unique vertex algebra  structure on the induced $\g^0$-module $V_{\g^0}$ as defined in \eqref{intro:vg0}
such that  $\bm{1}$ is the vacuum vector
 and  $U$ is a generating set with $Y(u,z)=u(z)$ for $u\in U$.
 Furthermore,
since $Y(d^n u,z)=Y(u_{-n-1}{\mathbf 1},z)$ for $n\in \N$ and $u\in U$ (\cite{DLM}), we see that
\begin{align}\label{eq:genVg0}
\{a^{\alpha,0}\mid a\in A, \alpha\in \Gamma\}\end{align}
is also a generating set of $V_{\g^0}$ with
\begin{align}\label{eq:genVg0act}
Y(a^{\alpha,0},z)=a^{\alpha,0}(z).
\end{align}

In what follows we define a $\Gamma$-action $R$ on $V_{\g^0}$ so that $V_{\g^0}$ becomes a $(\Gamma,\epsilon)$-vertex algebra as defined in Introduction.
First, for $\lambda\in \Gamma$, we define a linear automorphism $\bar{R}_\lambda$ on $\bar{\g}^0$ by
\[\bar R_\lambda(\bar{a}^{\alpha,0}(m))=\lambda^{(m+1)(\epsilon-1)}\bar{a}^{\alpha\lambda^{-1},0}(m)\]
for $a\in A,\alpha\in\Gamma$ and $m\in\Z$,  or equivalently, by
  $$\bar R_\lambda(\bar a^{\alpha,0}(z))=\bar{a}^{\alpha\lambda^{-1},0}(\lambda^{1-\epsilon}z)$$
  for $a\in A,\alpha\in\Gamma$.

 \begin{lemt}\label{lem:barR} For any $\lambda\in \Gamma$, one has  $\bar{R}_\lambda(\bar{\g}^0_0)= \bar{\g}^0_0$.
\end{lemt}
\begin{proof} Recall that  $\bar{\g}^0_0$ is spanned by the
coefficients of the generating functions in $\ker \bar{\psi}^0$.
It suffices to show that for any $u(z)\in\ker\bar{\psi}^0$, $\bar R_\lambda(u(z))=v(\lambda^{1-\epsilon} z)$ for some $v(z)\in\ker\bar{\psi}^0 $.

 For any $u(z)\in\ker\bar{\psi}^0$,
write $u(z)=\sum_{i=1}^k\mu_i\(\frac{\partial}{\partial z}\)^{n_i}\overline{a_i}^{\alpha_i,0}(z)$, where $\mu_i\in\C,n_i\in\N, a_i\in A$ and $\alpha_i\in\Gamma$. By definition we have
\[\bar R_\lambda(u(z))=\sum_{i=1}^k\mu_i \(\frac{\partial}{\partial z}\)^{n_i}\overline{a_i}^{\alpha_i\lambda^{-1},0}(\lambda^{1-\epsilon} z)=v(\lambda^{1-\epsilon} z),\]
where $v(z)=\sum_{i=1}^k\mu_i\lambda^{n_i(1-\epsilon)}\(\frac{\partial}{\partial z}\)^{n_i}\overline{a_i}^{\alpha_i\lambda^{-1},0}(z)$. Furthermore, since
\[\bar{\psi}^0(u(z))=\sum_{i=1}^k\mu_i\pd{z}{\epsilon}^{n_i}a_i(\alpha_iz)=0,\] we obtain
\[\bar{\psi}^0(v(z))=\sum_{i=1}^k\mu_i\lambda^{n_i(1-\epsilon)}\pd{z}{\epsilon}^{n_i}a_i(\alpha_i\lambda^{-1}z) =\bigg(\sum_{i=1}^k\mu_i\pd{w}{\epsilon}^{n_i}a_i(\alpha_iw)\bigg)\Big|_{w=\lambda ^{-1}z}=0.\]
Thus $v(z)\in\ker\bar{\psi}^0$, as desired.
\end{proof}
 By Lemma \ref{lem:barR}, $\bar{R}_\lambda$ induces a linear automorphism, say $R_\lambda$, on $\g^0$ such that
 \[ R_\lambda(a^{\alpha,0}(m))=\lambda^{(m+1)(\epsilon-1)}a^{\alpha\lambda^{-1},0}(m)\]
for $a\in A,\alpha\in\Gamma$ and $m\in\Z$.
\begin{lemt}\label{lem:defRlambda} For every $\lambda\in \Gamma$, $R_\lambda$ is a Lie automorphism of $\g^0$ which preserves  $\g^0_+$.
\end{lemt}
\begin{proof}The assertion that $R_\lambda(\g^0_+)=\g^0_+$ is obvious. For $a,b\in A,\alpha,\beta\in\Gamma$ we have
\begin{equation*}\begin{split}
&\,[ R_\lambda\big( a^{\alpha,0}(z)),R_\lambda\big(b^{\beta,0}(w)\big)]=[ a^{\alpha\lambda^{-1},0}(\lambda^{1-\epsilon} z),b^{\beta\lambda^{-1},0}(\lambda^{1-\epsilon} w)]\\
=&\,(\alpha\lambda^{-1})^{\epsilon-1}\sum_{\gamma\in \Gamma}\sum_{i,j\geq 0}\beta^{(i+j)(\epsilon-1)}\(\frac{\partial}{\partial w}\)^{j}(a_{(\alpha\beta^{-1},\gamma,i,j)}b)^{\gamma\beta\lambda^{-1},0}(\lambda^{1-\epsilon}  w)\Del{w,0}{i}{\lambda^{1-\epsilon} z,\lambda^{1-\epsilon} w}\\
=&\,\alpha^{\epsilon-1}\sum_{\gamma\in \Gamma}\sum_{i,j\geq 0}\beta^{(i+j)(\epsilon-1)}\(\frac{\partial}{\partial w}\)^{j}(a_{(\alpha\beta^{-1},\gamma,i,j)}b)^{\gamma\beta\lambda^{-1},0}(\lambda^{1-\epsilon}  w)\Del{w,0}{i}{z, w}\\
=&\,R_\lambda\big([a^{\alpha,0}(z),b^{\beta,0}(w)]\big).
\end{split}\end{equation*}
This proves that $R_\lambda$ is a Lie automorphism of $\g^0$.
\end{proof}

In view of Lemma \ref{lem:defRlambda},  $R_\lambda$ extends (uniquely) to
 an associative algebra automorphism of $\U(\g^0)$ and  preserves its left ideal  $\U(\g^0)\g^0_+$.
 Thus it
 induces a linear automorphism on $V_{\g^0}\cong \U(\g^0)/\U(\g^0)\g^0_+$, which we still call $R_\lambda$, such that
 \begin{align}\label{eq:rlambdarelation}
 R_\lambda(g.v)=R_\lambda(g).R_\lambda(v)
 \end{align}
 for $g\in \g^0$ and $v\in V_{\g^0}$.
In particular, we obtain a linear map
\[R:\Gamma\rightarrow \mathrm{GL}(V_{\g^0}),\quad \lambda\mapsto R_\lambda.\]
Let $a\in A$ and $\alpha,\lambda\in \Gamma$.
Recall that $a^{\alpha,0}=a^{\alpha,0}(-1)\vac\in V_{\g^0}$.
By applying \eqref{eq:rlambdarelation}, we have
\begin{align}\label{eq:ractongen}
R_\lambda(a^{\alpha,0})=R_\lambda(a^{\alpha,0}(-1)\vac)
=R_\lambda(a^{\alpha,0}(-1))\vac
=a^{\alpha\lambda^{-1},0}(-1)\vac
=a^{\alpha\lambda^{-1},0}.
\end{align}

\begin{prpt}\label{prop:geva}
 $(V_{\g^0},R)$ is a $(\Gamma,\epsilon)$-vertex algebra.
\end{prpt}
\begin{proof}
Fix a $\lambda\in\Gamma$, and set $S=\{v\in V_{\g^0}\mid R_\lambda Y(v,z) R_\lambda^{-1}=Y(R_\lambda v, \lambda^{1-\epsilon}z)\}$.
We will prove the proposition by verifying that $S=V_{\g^0}$.
Let $a\in A,\alpha\in\Gamma$ and $u\in V_{\g^0}$.
From \eqref{eq:genVg0act} and \eqref{eq:rlambdarelation} we have
\begin{align*}
&R_\lambda((a^{\alpha,0})_mu)=R_\lambda(a^{\alpha,0}(m).u)\\
=&R_\lambda(a^{\alpha,0}(m)).(R_\lambda u)
=\lambda^{(m+1)(\epsilon-1)}a^{\alpha\lambda^{-1},0}(m).(R_\lambda u)\\
=&\lambda^{(m+1)(\epsilon-1)}(a^{\alpha\lambda^{-1},0})_m(R_\lambda u)
=\lambda^{(m+1)(\epsilon-1)}(R_\lambda a^{\alpha,0})_m(R_\lambda u).
\end{align*}
This implies that the generating set \eqref{eq:genVg0} of $V_{\g^0}$ lies in $S$.
Thus it suffices to prove that $S$ is a vertex subalgebra of $V_{\g^0}$.

It is clear that the vacuum vector $\vac\in S$. Let $u,v\in S$, $\nu\in V_{\g^0}$ and $m,n\in \Z$.
Then it follows from the Jacobi identity that
\begin{align*}
&\,R_\lambda\((u_mv)_n\nu\)\\
=&\,\sum_{i\geq 0}(-1)^i\binom{m}{i}\Big(R_\lambda\(u_{m-i}(v_{n+i}\nu)\)-(-1)^mR_\lambda\(v_{m+n-i}(u_{i}\nu)\)\Big)\\
=&\,\lambda^{(m+n+2)(\epsilon-1)}\sum_{i\geq 0}(-1)^i\binom{m}{i}\Big((R_\lambda u)_{m-i}\big((R_\lambda v)_{n+i}(R_\lambda \nu)\big)\\
&-(-1)^m(R_\lambda v)_{m+n-i}\big((R_\lambda u)_{i}(R_\lambda \nu)\big)\Big)\\
=&\,\lambda^{(m+n+2)(\epsilon-1)}\big((R_\lambda u)_m(R_\lambda v)\big)_n(R_\lambda \nu)=\lambda^{(n+1)(\epsilon-1)}\big(R_\lambda (u_m v)\big)_n(R_\lambda \nu).
\end{align*}
 This gives that $u_mv\in S$,  so $S$ is a vertex subalgebra of $V_{\g^0}$, as desired.
\end{proof}

 Theorem \ref{thm:main2}\,(I) follows from Proposition \ref{prop:geva}, \eqref{eq:genVg0act} and \eqref{eq:ractongen}.

\subsection{Basics on $\Gamma$-equivariant $\phi_\epsilon$-coordinated quasi modules}
Before proving the second part of  Theorem \ref{thm:main2},
 we recall some basic properties of $\Gamma$-equivariant $\phi_\epsilon$-coordinated quasi modules
for a $(\Gamma,\epsilon)$-vertex algebra.

Let $W$ be a vector space, and set \[\E(W)=\Hom(W,W((z)))\subset(\End W)[[z,z^{-1}]].\]
We define a $\Gamma$-action $\mathfrak{R}$ on $\E(W)$ by
\begin{align*}
\mathfrak{R}:\Gamma\rightarrow \mathrm{GL}(\E(W)),\quad \lambda\mapsto (\mathfrak{R}_\lambda:a(z)\mapsto a(\lambda^{-1}z)).
\end{align*}
Let $(a(z),b(z))$ be a $\Gamma$-quasi local pair in $\E(W)$ in the sense that
 there exists a polynomial $q(z)\in\C[z]$ whose roots lie in $\Gamma$ such that
\begin{eqnarray}\label{gammma-quasi-local}
q(z_1/z_2)\ [a(z_1),b(z_2)]=0.
\end{eqnarray}
We define the operation (cf. \cite{Li-CMP})
$$Y_\E^\epsilon(a(z),z_0)b(z)=\sum_{n\in\Z}a(z)_{n}^\epsilon b(z)z_0^{-n-1}\in\E(W)((z_0))$$
by the following rule:
\[Y_\E^\epsilon(a(z),z_0)b(z)=q(\phi_\epsilon(z,z_0)/z)^{-1}(q(z_1/z)a(z_1)b(z))\mid_{z_1=\phi_\epsilon(z,z_0)},\]
where $\phi_\epsilon(z,z_0)=e^{z_0z^\epsilon\frac{d}{dz}}z$.
The definition of $Y_\E^\epsilon$ does not depend on the choice of $q(z)$.
A $\Gamma$-quasi local subspace $U$ of $\E(W)$ is said to be {\em $Y_\E^\epsilon$-closed} if
$a(z)^\epsilon_{n} b(z)\in U$ for $a(z),b(z)\in U$ and $n\in\Z$.

We have the following results from \cite{CLTW}.

\begin{prpt}\label{prop:<S>}
Let $S$ be a $\Gamma$-stable and $\Gamma$-quasi local subset of $\E(W)$.
Then there is a  smallest $Y_\E^\epsilon$-closed $\Gamma$-quasi local subspace of $\E(W)$ which contains $S$ and $1_W$, denoted by $\<S\>_\epsilon$,
such that $\<S\>_\epsilon$ is $\Gamma$-stable and $(\<S\>_\epsilon,Y_\E^\epsilon,1_W,\mathfrak{R})$ is a $(\Gamma,\epsilon)$-vertex algebra. Moreover,  $W$ is a faithful
$\Gamma$-equivariant $\phi_\epsilon$-coordinated quasi $\langle S\rangle_\epsilon$-module with $Y_W^\epsilon(a(z),z_0)=a(z_0)$ for $a(z)\in \langle S\rangle_\epsilon$.
\end{prpt}

For a vertex algebra $V$, let $\D$ be the canonical derivation on $V$ defined by $\D v=v_{-2}\vac$ for $v\in V$.
The following result are from \cite{Li-CMP} and \cite{CLTW}.
\begin{lemt}
Let  $(W,Y_W^\epsilon)$ be  a $\Gamma$-equivariant  $\phi_\epsilon$-coordinated quasi module for a $(\Gamma,\epsilon)$-vertex algebra $(V,R)$.
Then we have for $v\in V, n\in \N$,
\begin{equation}\label{Y_W(D^nv,z)}\begin{split}
Y_W^\epsilon(v_{-n-1}\vac,z)=\frac{1}{n!}Y_W^\epsilon(\D ^nv,z)=\frac{1}{n!}\pd{z}{\epsilon}^nY_W^\epsilon(v,z),
\end{split}\end{equation}
and  for $u,v\in V$,
\begin{equation}\label{equi-commutator}\begin{split}
&[Y_W^\epsilon(u,z),Y_W^\epsilon(v,w)]
=\sum_{\lambda\in\Gamma}\sum_{i\geq0}\lambda^{1-\epsilon}Y_W\big((R_{\lambda^{-1}}u)_iv,w\big)\Del{w,\epsilon}{i}{z,\lambda w}.
\end{split}\end{equation}
\end{lemt}
\subsection{Proof of Theorem \ref{thm:main2} (II)}
In this subsection we prove the second part of  Theorem \ref{thm:main2}.

We say that a $\g^{\epsilon}[\Gamma]$-module $W$ is {\em restricted} if
for any $a\in A$ and $\alpha\in \Gamma$, $a^{\alpha}(z)\in \E(W)$,
recalling the generating function $a^{\alpha}(z)$ defined in \eqref{a^alpha(z)}.
Note that  a restricted $\g$-module is naturally a restricted $\g^\epsilon[\Gamma]$-module with $a^{1}(z)=a( z)$ for $a\in A$
(see Theorem \ref{thm:main1}\,(II)).

\begin{prpt}\label{prop:equi-mod}
The restricted $\g^\epsilon[\Gamma]$-modules $W$ are exactly  the
$\Gamma$-equivariant $\phi_\epsilon$-coordinated quasi $V_{\g^0}$-modules $(W,Y_W^\epsilon)$ with
$a^{\alpha}(z)=Y_W^{\epsilon}(a^{\alpha,0},z)$ for $ a\in A$ and $\alpha\in \Gamma$.
\end{prpt}
\begin{proof}
Let $(W,Y_W^\epsilon)$ be a $\Gamma$-equivariant $\phi_\epsilon$-coordinated quasi $V_{\g^0}$-module.
Let $\bar \eta:\bar\g^\epsilon\rightarrow\te{End}(W)$ be a linear map  determined by
\[ \bar a^{\alpha,\epsilon}(z)\mapsto Y_W^{\epsilon}(a^{\alpha,0},z)\quad (a\in A,\alpha\in\Gamma).\]
 In what follows we  show that $\bar\eta$ induces an action of $\g^\epsilon[\Gamma]$ on $W$.

Set
  \[v(z)=\sum_{i=1}^k \mu_i \(z^\epsilon\frac{\partial}{\partial z}\)^{n_i} \overline{a_i}^{\alpha_i,\epsilon}(z)\in \ker \bar{\psi}^\epsilon
 \quad (\mu_i\in \C, n_i\in \N, a_i\in A, \alpha_i\in \Gamma).\]
 Note that we have
 \[\bar\psi^0\(\sum_{i=1}^k \mu_i \(\frac{\partial}{\partial z}\)^{n_i} \overline{a_i}^{\alpha_i,0}(z)\)=
 \sum_{i=1}^k  \mu_i \(z^\epsilon\frac{\partial}{\partial z}\)^{n_i} a_i(\alpha_i z)=\bar\psi^\epsilon(v(z))=0.\]
 This implies that in $\g^0[[z,z^{-1}]]$:
\[\sum_{i=1}^k \mu_i \(\frac{\partial}{\partial z}\)^{n_i} a_i^{\alpha_i,0}(z)=0.\] In particular, the constant term $\sum_{i=1}^k\mu_i n_i!a_i^{\alpha_i,0}(-n_i-1)=0.$
Then  from \eqref{Y_W(D^nv,z)}, we have
\[\bar{\eta}(v(z))=\sum_{i=1}^k\mu_i\pd{z}{\epsilon}^{n_i}Y_W^\epsilon(a_i^{\alpha_i,0},z)=
Y_W^\epsilon\(\sum_{i=1}^k\mu_i n_i!(a_i^{\alpha_i,0})_{-n_i-1}{\bf 1},z\)=0.\]
Thus $\bar{\eta}$ factors through the subspace $\bar\g^\epsilon_0$ and yields  a linear map $\eta:\g^\epsilon=\bar\g^\epsilon/\bar\g^\epsilon_0\rightarrow\te{End}(W)$ such that
\[\quad  a^{\alpha,\epsilon}(z)\mapsto Y_W^{\epsilon}(a^{\alpha,0},z)\quad (a\in A,\alpha\in\Gamma).\]
Furthermore, by Definition \ref{def:geva} (2) and  \eqref{eq:ractongen} we obtain
  $$Y_W^\epsilon( a ^{\alpha,0},z)=Y_W^\epsilon(R_{\lambda^{-1}} a ^{\alpha\lambda^{-1} ,0},z)=Y_W^\epsilon(a^{\alpha\lambda^{-1},0},\lambda z).$$
Then $\eta$ factors through the subspace $\g^\epsilon_\Gamma$ and hence induces a $\g^\epsilon[\Gamma]$-action on $W$ through
\[\eta_\Gamma:\g^\epsilon[\Gamma]=\g^\epsilon/\g^\epsilon_\Gamma\rightarrow\te{End}(W)  \te{ determined by } a^{\alpha}(z)\mapsto Y_W^{\epsilon}(a^{\alpha,0},z)\quad (a\in A,\alpha\in\Gamma).\]

Now we prove that $(W,\eta_\Gamma)$ is a representation of $\g^\epsilon[\Gamma]$.
 From \eqref{equi-commutator}, we have
\begin{equation*}\begin{split}
&[Y_W^\epsilon(a^{\alpha,0},z),Y_W^\epsilon(b^{\beta,0},w)]=\sum_{\lambda \in\Gamma}\sum_{i\geq0} \lambda^{1-\epsilon} Y^\epsilon_W\big(a^{\lambda\alpha,0}(i)b^{\beta,0}(-1)\vac,w\big)\Del{w,\epsilon}{i}{z,\lambda w}
\end{split}\end{equation*}
for any $a,b\in A,\alpha,\beta\in\Gamma$. Note that $a^{\lambda\alpha,0}(i)b^{\beta,0}(-1)\vac=[a^{\lambda\alpha,0}(i),b^{\beta,0}(-1)]\vac$ for $i\in\N$.
Then  from \eqref{untwisted-lb},  \eqref{universal-twisted-lb} and \eqref{Y_W(D^nv,z)}, we have
\begin{equation*}\begin{split}
&[\eta_\Gamma(a^{\alpha}(z)),\eta_\Gamma(b^{\beta}(w))]=[Y_W^\epsilon(a^{\alpha,0},z),Y_W^\epsilon(b^{\beta,0},w)]\\
=&\sum_{\lambda,\gamma \in\Gamma}\sum_{i,j\geq0} \alpha^{\epsilon-1}\beta^{(i+j)(\epsilon-1)} j!Y^\epsilon_W
\big((a_{(\lambda\alpha\beta^{-1},\gamma,i,j)}b)^{\gamma\beta,0}(-j-1)\vac,w\big)
\Del{w,\epsilon}{i}{z,\lambda w}\\
=&\sum_{\lambda,\gamma \in\Gamma}\sum_{i,j\geq0} \alpha^{\epsilon-1}\beta^{(i+j)(\epsilon-1)}\pd{w}{\epsilon}^j
Y^\epsilon_W\big((a_{(\lambda\alpha\beta^{-1},\gamma,i,j)}b)^{\gamma\beta,0},w\big)
\Del{w,\epsilon}{i}{z,\lambda w}\\
=&\eta_\Gamma\big([a^{\alpha}(z),b^{\beta}(w)]_\Gamma\big).
\end{split}\end{equation*}
Thus $W$ is a (restricted) $\g^\epsilon[\Gamma]$-module with $a^{\alpha}(z)=Y_W^{\epsilon}(a^{\alpha,0},z)$ for $ a\in A, \alpha\in \Gamma$.

Conversely, let $W$ be a restricted $\g^\epsilon[\Gamma]$-module.
The commutator \eqref{universal-twisted-lb} implies that
\[S=\{a^{\alpha}(z)\mid a\in A,\alpha\in\Gamma\}\]
 is a $\Gamma$-quasi local subset of $\E(W)$. Furthermore, $S$ is $\Gamma$-stable as $a^{\alpha}(\lambda z)=a^{\alpha\lambda}(z)$ for $a\in A,\alpha,\lambda\in\Gamma$.
Then by applying Proposition \ref{prop:<S>}, we obtain that $(\<S\>_\epsilon,Y_\E^\epsilon,1_W,\mathfrak{R})$ is a $(\Gamma,\epsilon)$-vertex algebra and $W$ is a faithful
$\Gamma$-equivariant $\phi_\epsilon$-coordinated quasi $\langle S\rangle_\epsilon$-module with $Y_W^\epsilon(u(z),z_0)=u(z_0)$ for $u(z)\in \langle S\rangle_\epsilon$.
For  $a,b\in A$ and $\alpha,\beta\in\Gamma$, by \eqref{equi-commutator}  we have
\begin{equation*}\begin{split}
&[Y_W^\epsilon(a^{\alpha }( z),z_1),Y_W^\epsilon(b^{\beta }( z),z_2)]
=\sum_{\lambda\in\Gamma}\sum_{i\geq0}\lambda^{1-\epsilon} Y_W^\epsilon\(a^{\lambda\alpha }( z)_i^\epsilon b^{\beta }( z) ,z_2\)\Del{z_2,\epsilon}{i}{z_1,\lambda z_2}.
\end{split}\end{equation*}
Meanwhile, recall from \eqref{universal-twisted-lb},
\begin{align*}
&[Y_W^\epsilon(a^{\alpha }( z),z_1),Y_W^\epsilon(b^{\beta }( z),z_2)]=[a^{\alpha}(z_1),b^{\beta}(z_2)]_\Gamma\\
=&\,\sum_{\lambda,\gamma\in \Gamma}\sum_{i,j\geq 0}\alpha^{\epsilon-1}\beta^{(i+j)(\epsilon-1)}\(\pd{ z_2}{\epsilon}^{j}(a_{(\lambda\alpha\beta^{-1},\gamma,i,j)}b)^{\gamma\beta}(z_2)\)\Del{z_2,\epsilon}{i}{ z_1,\lambda z_2}.
\end{align*}
Thus we have
\begin{equation}\label{3-1}\begin{split}
a^{\alpha }( z)_i^\epsilon b^{\beta }( z) =\alpha^{\epsilon-1}\beta^{i(\epsilon-1)}\sum_{\gamma\in \Gamma}\sum_{j\geq 0}\beta^{j(\epsilon-1)} \pd{ z}{\epsilon}^{j}(a_{(\alpha\beta^{-1},\gamma,i,j)}b)^{\gamma\beta}( z). \end{split}\end{equation}

Let  $u(z)\in\langle S\rangle_\epsilon$. It follows from \eqref{Y_W(D^nv,z)} that
\[\D (u(z_0))=Y_W^\epsilon(\D(u(z)),z_0)=\pd{z_0}{\epsilon}Y_W^\epsilon(u(z),z_0)=\pd{z_0}{\epsilon}u(z_0).\]
This implies that (\cite[Proposition 3.1.18]{LL})
\[Y_\E^\epsilon\(\pd{z}{\epsilon}u(z),z_0\)=Y_\E^\epsilon\(\D(u(z)),z_0\)=\frac{\partial }{\partial z_0}Y_\E^\epsilon\( u(z),z_0\).\]
Combining this with \eqref{3-1}, from Borcherds commutator formula (\cite[(3.1.8)]{LL}),  we have
\begin{align*}
&\,[Y_\E^\epsilon(a^{\alpha}(z),z_1),Y_\E^\epsilon(b^{\beta}(z),z_2)]\\
=&\,\sum_{i\geq0}Y_\E^\epsilon(a^{\alpha}(z)_i^\epsilon b^{\beta}(z),z_2)\Del{z_2,0}{i}{ z_1, z_2}\\
=&\,\sum_{i,j\geq0}\sum_{\gamma\in \Gamma}\alpha^{\epsilon-1}\beta^{(i+j)(\epsilon-1)}Y_\E^\epsilon\(\pd{ z}{\epsilon}^{j}(a_{(\alpha\beta^{-1},\gamma,i,j)}b)^{\gamma\beta}( z) ,z_2\)\Del{z_2,0}{i}{ z_1, z_2}\\
=&\,\sum_{i,j\geq 0}\sum_{\gamma\in \Gamma}\alpha^{\epsilon-1}\beta^{(i+j)(\epsilon-1)} \(\frac{\partial}{\partial z_2}\)^{j}Y_\E^\epsilon\( (a_{(\alpha\beta^{-1},\gamma,i,j)}b)^{\gamma\beta}( z),z_2\)\Del{z_2,0}{i}{ z_1, z_2}.
\end{align*}
Comparing this with  \eqref{untwisted-lb}, we deduce that  $\langle S\rangle_\epsilon$ is a $\g^0$-module  with the action
\[a^{\alpha,0}(z_0)=Y_\E^\epsilon(a^{\alpha}(z),z_0)\quad \te{for }a\in A,\alpha\in\Gamma.\]

Note that $\langle S\rangle_\epsilon$ is generated by $1_W$ as a $\g^0$-module and  $\g^0_+1_W=0$.
From the universal property of $V_{\g^0}$, there is
 a (unique) $\g^0$-module homomorphism
$\pi:V_{\g^0}\rightarrow\langle S\rangle_\epsilon$ such that $\pi(\vac)=1_W$.
For $a\in A,\alpha\in\Gamma$ and $v\in V_{\g^0}$, it follows that
\[\pi(Y(a^{\alpha,0},z_0)v)=\pi\(a^{\alpha,0}(z_0)v\)=Y_\E^\epsilon\(a^{\alpha }(z),z_0\)\pi(v)=Y_\E^\epsilon\(\pi(a^{\alpha,0}),z_0\)\pi(v),\]
 where we used the fact that
\begin{align*}
\pi(a^{\alpha,0})=&\,\pi\(\Res_{z_0}z_0^{-1}Y(a^{\alpha,0},z_0)\vac\)=\pi\(\Res_{z_0}z_0^{-1}a^{\alpha,0}(z_0)\vac\)\\
=&\,\Res_{z_0}z_0^{-1}Y_\E^\epsilon(a^{\alpha }(z),z_0)1_W=a^{\alpha }(z).
\end{align*}
This implies that
 $\pi$ is also a vertex algebra homomorphism. Note that $\{a^{\alpha,0}\mid a\in A,\alpha\in\Gamma\}$ is a generating set of $V_{\g^0}$.
 Furthermore, for $a,b\in A$ and $\alpha,\beta,\lambda\in\Gamma$, we have
\begin{equation*}\begin{split}
&\mathfrak{R}_\lambda\circ\pi\(Y(a^{\alpha,0},z_0)b^{\beta,0}\)=\mathfrak{R}_\lambda\(Y_\E^\epsilon(a^\alpha(z),z_0)b^\beta(z)\)\\
=&Y_\E^\epsilon\big(a^{\alpha}(\lambda^{-1} z),\lambda^{1-\epsilon} z_0\big)b^{\beta}(\lambda^{-1}z)=Y_\E^\epsilon\big(a^{\alpha\lambda^{-1}}( z),\lambda^{1-\epsilon} z_0\big)b^{\beta\lambda^{-1}}(z)\\
=&\pi\big(Y\big(a^{\alpha\lambda^{-1},0},\lambda^{1-\epsilon} z_0\big)b^{\beta\lambda^{-1},0}\big)=\pi\circ R_\lambda\(Y(a^{\alpha,0}, z_0)b^{\beta,0}\).
 \end{split}\end{equation*}
This says that $\pi$ is a $(\Gamma,\epsilon)$-vertex algebra homomorphism in the sense that  $\mathfrak{R}_\lambda\circ\pi=\pi\circ R_\lambda$ for $\lambda\in\Gamma$.
Thus, via the homomorphism $\pi$, $W$ becomes a $\Gamma$-equivariant $\phi_\epsilon$-coordinated quasi $V_{\g^0}$-module with $Y_W^{\epsilon}(a^{\alpha,0},z)=a^{\alpha}(z)$ for $ a\in A, \alpha\in \Gamma$.
\end{proof}

Finally, it is clear that Theorem \ref{thm:main2}\,(II) follows from Proposition \ref{prop:equi-mod} and Theorem \ref{thm:main1}\,(II).

\section{Examples}
In this section, we give five typical examples of quasi vertex Lie algebras: (i) the twisted affine Lie algebras, (ii) the quantum torus Lie algebras, (iii) the $q$-Heisenberg Lie algebras, (iv) the Virasoro-like algebras and (v) the Klein bottle Lie algebras. We shall use  Theorem  \ref{thm:main2} to associate them with vertex algebras.

\subsection{Twisted affine Lie algebras}
Let $\mathfrak{b}$ be a Lie algebra equipped with an invariant symmetric bilinear form $\<\cdot,\cdot\>$. Denote by
 \[\widehat{\CL}(\mathfrak{b})=\(\mathfrak{b}\ot \C[t,t^{-1}]\)\op\C\mk\]
the affine Lie algebra associated to the pair $(\mathfrak{b},\<\cdot,\cdot\>)$, where $\mk$ is central and for $a,b\in\mathfrak{b}, m,n\in \Z$,
 \begin{align}\label{affine-lb0}
[a\ot t^m,b\ot t^n]=[a,b]\ot t^{m+n}+\delta_{m,-n}\<a,b\>m\mk.
\end{align}
In terms of generating functions
$a(z)=\sum_{n\in\Z}(a\ot t^n) z^{-n-1}\ (a\in\mathfrak{b})$,
the commutator \eqref{affine-lb0} can be rewritten as follows
 \begin{align}\label{affine-lb}
[a(z),b(w)]=[a,b](w) z^{-1}\delta\left(\frac{w}{z}\right)+\<a,b\>\mk \frac{\partial}{\partial w} z^{-1}\delta\left(\frac{w}{z}\right) .
\end{align}
This implies that $\widehat{\CL}(\mathfrak{b})$ is a vertex Lie algebra and we have the  affine vertex algebra
\begin{align}\label{eq:affva}V_{\widehat{\CL}(\mathfrak{b})}=\U(\widehat{\CL}(\mathfrak{b}))\otimes_{\U(\mathfrak{b}\ot \C[t])}\C,\end{align}
where ${\bf 1}=1\ot 1$ is the vacuum vector and $Y(a,z)=a(z)$ for $a\in \mathfrak{b}$.
As usual, we  identify $\mathfrak{b}$ as a subspace of $V_{\widehat{\CL}(\mathfrak{b})}$ through the  map $a\mapsto (a\ot t^{-1})\ot1$ for $a\in \mathfrak{b}$.

Let $\sigma$ be a finite order automorphism of $\g$ that preserves the form $\<\cdot,\cdot\>$.
Denote by $T$ the order of $\sigma$.
Then we have the $\sigma$-twisted affine Lie algebra (\cite{K})
\[\widehat{\CL}(\mathfrak{b},\sigma)=\bigoplus_{k=0}^{T-1}\(\mathfrak{b}_{(k)}\ot t^{k}\C[t^T,t^{-T}]\)\oplus\C\mk\subset \widehat{\CL}(\mathfrak{b}),\]
where  $\mathfrak{b}_{(k)}=\{a\in\mathfrak{b}\mid \sigma(a)=q^k a\}$ and $q=e^{\frac{2\pi \sqrt{-1}}{T}}$. Let $\epsilon$ be an integer.
For $a\in \mathfrak{b}_{(k)}$, set
\[a_\sigma(z)=\sum_{n\in\Z}(a\ot t^{k+nT}) z^{-k-nT+\epsilon-1}.\]
For $a\in\mathfrak{b}_{(k)}$ and $b\in \mathfrak{b}_{(l)}$, it is straightforward to check that
\begin{align*}
&[a_\sigma(z),b_\sigma(w)]
=\sum_{s=0}^{T-1}\frac{q^{-ks}}{T}\([a,b]_\sigma(w) z^{\epsilon-1}\delta\left(\frac{q^s w}{z}\right) +\<a,b\>\mk \(w^\epsilon\frac{\partial}{\partial w}\) z^{\epsilon-1}\delta\left(\frac{q^s w}{z}\right)\).
\end{align*}
 This implies that
$(\widehat{\CL}(\mathfrak{b},\sigma),\CA,\epsilon)$ is a quasi vertex Lie algebra with
 $\Gamma_\sigma=\<q\>$ as the associated group, where
\[\CA=\{a_\sigma(z),(\mu\mk)(z):=\mu \mk\mid a\in\mathfrak{b}_{(k)},0\le k\le T-1,\mu\in\C\}.\]

Next we consider the maximality of $\widehat{\CL}(\mathfrak{b},\sigma)$.
 Recall that  $\widetilde{\widehat{\CL}(\mathfrak{b},\sigma)}$ is the complex vector space with a basis
\[\{\widetilde{a}(m)\mid a\in A, m\in\Z\},\]
where $A=\{a,\mu\mk\mid a\in\mathfrak{b}_{(k)},0\le k\le T-1,\mu\in\C\}$. Recall that  $\widetilde{a}(z)=\sum_{m\in \Z} \widetilde{a}(m)z^{-m+\epsilon-1}$ for $a\in A$.
Let $\CR$ be the set consisting of:
\begin{equation*}\begin{split}
&\widetilde{a+a'}(z)-\widetilde{a}(z)-\widetilde{a'}(z),\quad \widetilde{a}(q z)-q^{-k+\epsilon-1}\widetilde{a}(z),\\
 &\widetilde{\mu a}(z)-\mu\widetilde{a}(z),\quad
\widetilde{\mu\mk}(z)-\mu\widetilde{\mk}(z),\quad z^\epsilon\frac{\partial}{\partial z}\widetilde{\mk}(z),\\
\end{split}\end{equation*}
where $a,a'\in \mathfrak{b}_{(k)}, 0\le k\le T-1$ and $\mu\in\C$.
Consider the quotient space  $\widetilde{\widehat{\CL}(\mathfrak{b},\sigma)}/\widetilde{\widehat{\CL}(\mathfrak{b},\sigma)}_{\CR}$, where $\widetilde{\widehat{\CL}(\mathfrak{b},\sigma)}_\CR$ is spanned
by the coefficients of the generating functions in $\CR$. Note that
$\widetilde{\widehat{\CL}(\mathfrak{b},\sigma)}/\widetilde{\widehat{\CL}(\mathfrak{b},\sigma)}_{\CR}$ is spanned by $$\widetilde{a}(k+mT)+\widetilde{\widehat{\CL}(\mathfrak{b},\sigma)}_{\CR},\quad \widetilde{\mk}(\epsilon-1)+\widetilde{\widehat{\CL}(\mathfrak{b},\sigma)}_{\CR},$$
where $a\in \mathfrak{b}_{(k)}, 0\le k\le T-1$ and $m\in\Z$.
 There is a canonical surjective map from
$\widetilde{\widehat{\CL}(\mathfrak{b},\sigma)}/\widetilde{\widehat{\CL}(\mathfrak{b},\sigma)}_{\CR}$ to $\widehat{\CL}(\mathfrak{b},\sigma)$ (see \eqref{eq:canmaxiR}) defined by
\begin{align}\label{eq:maxi-twisted}
\widetilde{a}(k+mT)+\widetilde{\widehat{\CL}(\mathfrak{b},\sigma)}_{\CR}\mapsto a\ot t^{k+mT},\quad \widetilde{\mk}(\epsilon-1)+\widetilde{\widehat{\CL}(\mathfrak{b},\sigma)}_{\CR}\mapsto\mk
\end{align}
for $a\in \mathfrak{b}_{(k)}, 0\le k\le T-1$ and $m\in\Z$.
It is easy to check that \eqref{eq:maxi-twisted} is an isomorphism of vector spaces. Thus from Remark \ref{maximal-equi},  $\widehat{\CL}(\mathfrak{b},\sigma)$ is maximal.

For any $\zeta\in \Z$, from  Theorem \ref{thm:main1} we have a Lie algebra $\widehat{\CL}(\mathfrak{b},\sigma)^{\zeta}$. Note that
\[a^{\alpha,\zeta}(z)=\alpha^{-k+\epsilon-1}a^{1,\zeta}(z),\quad \mk^{\alpha,\zeta}(z)=\mk^{1,\zeta} (\zeta-1) \]
for $a\in\mathfrak{b}_{(k)},\ 0\le k\le T-1$ and $\alpha\in\Gamma$. Then $\widehat{\CL}(\mathfrak{b},\sigma)^{\zeta}$ is spanned by $\mk^{1,\zeta} (\zeta-1)$
and the coefficients of $a^{1,\zeta}(z)$ for  $a\in\mathfrak{b}_{(k)}$, $0\le k\le T-1$. Furthermore,   from \eqref{untwisted-lb} we have
\begin{equation*}\begin{split}
[a^{1,\zeta}(z),b^{1,\zeta}(w)]=&\frac{1}{T}\([a,b]^{1,\zeta}(w) z^{\zeta-1}\delta\left(\frac{w}{z}\right) +\<a,b\>\mk^{1,\zeta}(\zeta-1) \(w^\zeta\frac{\partial}{\partial w}\) z^{\zeta-1}\delta\left(\frac{w}{z}\right) \)
\end{split}\end{equation*}
for  $a\in\mathfrak{b}_{(k)},b\in \mathfrak{b}_{(l)}$, $0\le k,l\le T-1$.
Comparing this with \eqref{affine-lb}, it follows that $\widehat{\CL}(\mathfrak{b},\sigma)^\zeta$ is isomorphic to $\widehat{\CL}(\mathfrak{b})$  with
\[a^{1,\zeta}(z)\mapsto \frac{1}{T}z^\zeta a(z),\quad \mk^{1,\zeta}(\zeta-1)\mapsto \frac{1}{T}\mk\quad\te{for }a\in\mathfrak{b}_{(k)},\ 0\le k\le T-1.\]

From the isomorphism $\widehat{\CL}(\mathfrak{b},\sigma)^0\cong \widehat{\CL}(\mathfrak{b})$,
we can obtain the following result  from the Theorem \ref{thm:main2}  (cf. \cite{Li-adv,CLTW}).

\begin{prpt}  Let $\widehat{\CL}(\mathfrak{b},\sigma)$ be the twisted affine Lie algebra associated to the triple $(\mathfrak{b},\<\cdot,\cdot\>,\sigma)$ as above,  and let $\epsilon$ be an integer. Set $\Gamma_\sigma=\langle q\rangle$, where  $q=e^{\frac{2\pi \sqrt{-1}}{T}}$ and $T$ is the order of $\sigma$. Then there is a  $(\Gamma_\sigma,\epsilon)$-vertex algebra structure on  $V_{\widehat{\CL}(\mathfrak{b})}$
such that $R_q\(a\)=q^{k-\epsilon+1}a$ for $a\in \mathfrak{b}_{(k)}$, $0\le k\le T-1$.
Furthermore,
$\Gamma_\sigma$-equivariant $\phi_\epsilon$-coordinated quasi $V_{\widehat{\CL}(\mathfrak{b})}$-modules are exactly restricted $\widehat{\CL}(\mathfrak{b},\sigma)$-modules.
\end{prpt}

\subsection{Quantum torus Lie algebras} Let $N$ be a positive integer and let  $Q=(q_{ij})$ be an $(N+1)\times (N+1)$ matrix such that
$q_{ij}\in \C^\times$, $q_{ii}=1$ and $q_{ij}=q_{ji}^{-1}$ for $0\le i,j\le N$.
Let $\C_Q$ be the quantum torus  associated to $Q$ as defined in \cite{BGK}, that is,  $\C_Q$ is a unital associative algebra with $\C_Q=\C[t_0^{\pm1},t_1^{\pm1},\dots,t_N^{\pm1}]$
as a vector space and $t_it_j=q_{ij}t_j t_i$ for $0\le i,j\le N$.
For $\mb{m}=(m_1,\dots,m_N),\mb{n}=(n_1,\dots,n_N)\in\Z^N$, set \[\mb{t^m}=t_{1}^{m_1}\cdots t_{N}^{m_N},\quad \mb{q^m}=q_{10}^{m_1}\cdots q_{N0}^{m_N}\quad \te{and}\quad\sigma(\mb{m,n})=\prod_{1\le s\le k\le N}q_{ks}^{m_kn_s} .\]
Then $\mb{t^m}t_0=\mb{q^m}t_0\mb{t^m}$ and  $\mb{t^m}\mb{t^n}=\sigma(\mb{m,n})\mb{t^{m+n}}$ for $\mb{m,n}\in\Z^N$.

 Let  $\ell$ be any positive integer. View
$\fgl_\ell(\C)\ot \C_Q$
as a Lie algebra with commutator as its Lie bracket, and consider a one-dimensional central extension:
\[\widehat{\fgl_\ell}(\C_Q)=\(\fgl_\ell(\C)\ot \C_Q\)\oplus\C \mk,\]
where $\mk$ is central and
\begin{equation}\label{lb-glCq-fenliang}\begin{split}
[x\ot t_0^{m}\mb{t^m},y\ot t_0^{n}\mb{t^n}]
=&\,\sigma(\mb{m,n})\mb{q}^{n\mb{m}}xy\ot t_0^{m+n}\mb{t^{m+n}}-\sigma(\mb{n,m})\mb{q}^{m\mb{n}}yx\ot t_0^{m+n}\mb{t^{m+n}}\\
&+\delta_{m,-n}\delta_{\mb{m,-n}}\sigma(\mb{m,n})\mb{q}^{n\mb{m}}\mathrm{Tr}(xy)m\mk,
 \end{split}\end{equation}
where $x,y\in\fgl_\ell(\C)$, $m,n\in\Z$, $\mb{m,n}\in\Z^N$ and $\mathrm{Tr}$ denotes the trace form. Let $\epsilon$ be an integer. Set
\[(x\mb{t^m})(z)=\sum_{m\in\Z}x\ot t_0^{m}\mb{t^m}z^{-n+\epsilon-1}\quad(x\in\fgl_\ell(\C),\mb{m}\in\Z^N).\]
We rewrite \eqref{lb-glCq-fenliang} in terms of the generating functions:
\begin{equation}\label{lb-glCq}\begin{split}
[x\mb{t^m}(z),y\mb{t^n}(w)]&=(\mb{q^{m}})^{\epsilon-1}\mb{\sigma(m,n)}(xy\mb{t^{m+n}})(\mb{q^{-m}}w)  z^{\epsilon-1}\delta\left(\frac{\mb{q^{-m}} w}{z}\right)\\
&\quad-\mb{\sigma(n,m)}(yx\mb{t^{m+n}})(w) z^{\epsilon-1}\delta\left(\frac{\mb{q^{n}} w}{z}\right)\\
&\quad+\mb{\sigma(m,n)}\mathrm{Tr}(xy)\delta_{\mb{m,-n}}\mk\(w^\epsilon\frac{\partial}{\partial w}\) z^{\epsilon-1}\delta\left(\frac{\mb{q^{-m}} w}{z}\right) .
\end{split}\end{equation}
Similar to the analysis as the twisted affine Lie algebras, we see that $(\widehat{\fgl_\ell}(\C_Q),\CA,\epsilon)$ is a maximal  quasi vertex Lie algebra with the associated group $\Gamma_Q=\{\mb{q^m}\mid \mb{m}\in\Z^N\}$,
where $$\CA=\{(x\mb{t^m})(z),(\mu\mk)(z):=\mu\mk\mid x\in\fgl_\ell(\C),\mb{m}\in\Z^N,\mu\in\C\}.$$

For $1\le i,j\le \ell$, let
 $E_{i,j}\in\fgl_\ell(\C)$ be  the elementary matrix having $1$ in $(i,j)$-position and $0$ elsewhere.
It is routine to check that $\widehat{\fgl_\ell}(\C_Q)^{\zeta}$ (see Theorem \ref{thm:main1}) has a basis
 \[(E_{i,j}\mb{t^m})^{\alpha,\zeta}(m),\ \mk^{1,\zeta}(\zeta-1)\quad\te{ for } 1\le i,j\le \ell,\ \mb{m}\in\Z^N,\alpha\in\Gamma_Q,m\in\Z,\]
where $\mk^{1,\zeta}(\zeta-1)$ is central element and
\begin{align}\label{lb-glCq-zeta}\begin{split}
&[(E_{i,j}\mb{t^m})^{\alpha,\zeta}(z),(E_{i',j'}\mb{t^n})^{\beta,\zeta}(w)]\\
=&\,\delta_{\beta\alpha^{-1},\mb{q^{m}}}\delta_{j,i'}\beta^{\epsilon-1}\mb{\sigma(m,n)}(E_{i,j'}\mb{t^{m+n}})^{\alpha,\zeta}(w) z^{\zeta-1}\delta\left(\frac{w}{z}\right)\\
&-\delta_{\alpha\beta^{-1},\mb{q^{n}}}\delta_{j',i}\alpha^{\epsilon-1}\mb{\sigma(n,m)}(E_{i',j}\mb{t^{m+n}})^{\beta,\zeta}(w) z^{\zeta-1}\delta\left(\frac{w}{z}\right)\\
&+\delta_{\beta\alpha^{-1},\mb{q^{m}}}\delta_{j,i'}\delta_{j',i}(\alpha\beta)^{\epsilon-1}\mb{\sigma(m,n)}\delta_{\mb{m,-n}}\mk^{1,\zeta}(\zeta-1)\(w^\zeta\frac{\partial}{\partial w}\) z^{\zeta-1}\delta\left(\frac{w}{z}\right)
\end{split}\end{align}
for $1\le i,j,i',j'\le \ell,\alpha,\beta\in\Gamma_Q$ and $\mb{m,n}\in\Z^N$.

Let $\fgl_{\ell,Q}$ be a vector space with a basis \[E_{i,j}^{\mb{m},\alpha}\quad \te{for } 1\le i,j\le \ell,\ \mb{m}\in\Z^N,\alpha\in\Gamma_Q.\]
We define a  multiplication on $\fgl_{\ell,Q}$ by
\begin{align}\label{mul-gl-N-Gamma}
&E_{i,j}^{\mb{m},\alpha}\cdot E_{i',j'}^{\mb{n},\beta}=\delta_{\beta\alpha^{-1},\mb{q^{m}}}\delta_{j,i'}\mb{\sigma(m,n)}E_{i,j'}^{\mb{m+n},\alpha},
\end{align}
and define a symmetric bilinear form  on $\fgl_{\ell,Q}$ by
\begin{align}\label{form-gl-N-Gamma}
\<E_{i,j}^{\mb{m},\alpha}, E_{i',j'}^{\mb{n},\beta}\>=\delta_{\beta\alpha^{-1},\mb{q^{m}}}\delta_{j,i'}\delta_{j',i}\mb{\sigma(m,n)}\delta_{\mb{m,-n}},
\end{align}
 where $1\le i,j,i',j'\le \ell,\ \mb{m,n}\in\Z^N$ and $\alpha,\beta\in\Gamma_Q$.
 It is straightforward to check that $\fgl_{\ell,Q}$ is an associative algebra under the multiplication \eqref{mul-gl-N-Gamma}, and the form $\<\cdot,\cdot\>$ is (associative) invariant.
View $\fgl_{\ell,Q}$ as  a Lie algebra,
associated to the pair $(\fgl_{\ell,Q},\<\cdot,\cdot\>)$, we have an affine Lie algebra $\widehat{\CL}(\fgl_{\ell,Q})$.

By using \eqref{affine-lb} and \eqref{lb-glCq-zeta}-\eqref{form-gl-N-Gamma}, one can check that  $\widehat{\CL}(\fgl_{\ell,Q})$ is isomorphic to  $\widehat{\fgl_\ell}(\C_Q)^{\zeta}$ with the mapping  $\mk\mapsto\mk^{1,\zeta}(\zeta-1)$ and $E_{i,j}^{\mb{m},\alpha}\ot t^m\mapsto \alpha^{1-\epsilon}(E_{i,j}\mb{t^m})^{\alpha,\zeta}(m)$ for $1\le i,j\le \ell,\mb{m}\in\Z^N$, $\alpha\in\Gamma_Q$ and $m\in\Z$.
In particular, from isomorphism $\widehat{\CL}(\fgl_{\ell,Q})\cong \widehat{\fgl_\ell}(\C_Q)^{0}$, we have the following result from
Theorem \ref{thm:main2}.

\begin{prpt}\label{prop:quantumtours} There is a $(\Gamma_Q,\epsilon)$-vertex algebra structure on $V_{\widehat{\CL}(\fgl_{\ell,Q})}$ with
 $R_{\lambda}\(E_{i,j}^{\mb{m},\alpha}\)=\lambda^{1-\epsilon} E_{i,j}^{\mb{m},\alpha\lambda^{-1}}$
 for $1\le i,j\le \ell,\mb{m}\in\Z^N$ and $\alpha,\lambda\in\Gamma_Q.$
Furthermore, $\Gamma_Q$-equivariant $\phi_\epsilon$-coordinated quasi $V_{\widehat{\CL}(\fgl_{\ell,Q})}$-modules are exactly restricted $\widehat{\fgl_\ell}(\C_Q)$-modules.
\end{prpt}

\begin{remt} When $\epsilon=0$, $N=1$, and $q_{10}$ not a root of unity, Proposition \ref{prop:quantumtours} was obtained in \cite{Li3} (see also \cite{LTW}).
In this case, $\fgl_{\ell,Q}$  is isomorphic to $\fgl_\infty$ (\cite{Li3}).
\end{remt}

\begin{remt} In the case that $\g=\wh{\CL}(\mathfrak{b},\sigma)$ or $\widehat{\fgl_\ell}(\C_Q)$,
 for any $\epsilon\in\Z$ and certain group $\Gamma$,
 there is a  canonical quasi vertex Lie algebra structure on $\g$ and
 restricted $\g$-modules are exactly $\Gamma$-equivariant $\phi_\epsilon$-coordinated quasi $V_{\g^0}$-modules.
These results  also true for  $\g$  being $q$-Virasoro algebras (see \cite{GLTW1,GLTW2}) and unitary Lie algebras (see \cite{GW}).  Specifically, if we take $\epsilon=0$ or 1, these results are the main results in  \cite{GLTW1,GLTW2,GW}.
\end{remt}

\subsection{$q$-Heisenberg Lie algebras }
In this subsection, let $q$ be a nonzero complex number with $q\neq \pm 1$.
Consider the $q$-Heisenberg Lie algebra (cf. \cite{FR,Li-FMC})
\[H_q=\oplus_{m\in\Z}\C a(m)\oplus \C \mb{c},\]
where $\mb{c}$ is central and for $m,n\in \Z$,
\[[a(m),a(n)]=\frac{q^m-q^{-m}}{q-q^{-1}}\delta_{m,-n}\mb{c}.\]
Equivalently, by setting $a(z)=\sum_{m\in\Z}a(m)z^{-m}$, we have
 \[[a(z),a(w)]=\frac{1}{q-q^{-1}}\mb{c}\(\delta\left(\frac{qw}{z}\right) - \delta\left(\frac{q^{-1}w}{z}\right) \).\]
Then $(H_q,\CA,1)$ is a maximal quasi vertex Lie algebra with $\Gamma_q=\<q\>$ as the associated group, where $\CA=\{a(z),(\mu\mb{c})(z):=\mu\mb{c}\mid \mu\in\C\}$.
 Furthermore, the Lie algebra $H_q^{\zeta}$ has  a basis \[\{a^{\alpha,\zeta}(m),\ \mb{c}^{1,\zeta}(\zeta-1)\mid \alpha\in\Gamma_q, m\in\Z\}\]
such that $\mb{c}^{1,\zeta}(\zeta-1)$ is  central and for $\alpha,\beta\in\Gamma_q,m,n\in\Z$,
\begin{align*}
&[a^{\alpha,\zeta}(m+\zeta),a^{\beta,\zeta}(n+\zeta)]
=\frac{1}{q-q^{-1}}\(\delta_{\alpha\beta^{-1},q}-\delta_{\alpha\beta^{-1},q^{-1}}\)\delta_{m+n+\zeta+1,0}\mb{c}^{1,\zeta}(\zeta-1).
\end{align*}

Let $H$ be a vector space equipped with a basis  $\{b^\alpha\mid \alpha\in\Gamma\}$ and a skew-symmetric bilinear form $\<\cdot,\cdot\>$ such that
\begin{align}\label{4-4}
\<b^\alpha,b^\beta\>=\frac{1}{q-q^{-1}}\(\delta_{\alpha\beta^{-1},q}-\delta_{\alpha\beta^{-1},q^{-1}}\)\quad (\alpha,\beta\in\Gamma_q).
\end{align}
We associate a Heisenberg Lie algebra $\widehat{H}=(H\ot \C[t,t^{-1}])\oplus \C \mb{c}$ with $(H,\<\cdot,\cdot\>)$ such that
\begin{align*}
[\mb{c},\widehat{H}]=0\quad\text{and}\quad [b^\alpha\ot t^m,b^\beta\ot t^n]=\delta_{n+m+1,0}\<b^\alpha,b^\beta\>\mb{c}
\end{align*}
for $\alpha,\beta\in\Gamma_q$ and $m,n\in\Z$.
Note that $\wh{H}$ is a vertex Lie algebra and as in \eqref{eq:affva} we have the Heisenberg vertex algebra
\begin{align*}
V_{\widehat{H}}= \U(\widehat{H})\otimes_{\U(\widehat{H}_+)}\C,
\end{align*}
on which $Y(b^\alpha,z)=\sum_{m\in\Z}b^\alpha\ot t^m z^{-m-1}$ for $\alpha\in \Gamma_q$,
where $\widehat{H}_+=\sum_{\alpha\in \Gamma,m\ge 0}\C b^\alpha\ot t^m$, $\C$ is the trivial $\widehat{H}_+$-module and
$b^\alpha=(b^{\alpha}\ot t^{-1})\ot1\in V_{\widehat{H}}$.

Note that the Heisenberg Lie algebra  $\widehat{H}$ is isomorphic to $H_q^{0}$ with $\mb{c}\mapsto \mb{c}^{1,0}(-1)$ and $b^\alpha\ot t^m\mapsto a^{\alpha,0}(m)$ for $\alpha\in\Gamma_q, m\in\Z$. We have the following result by applying  Theorem \ref{thm:main2}, which was also obtained in \cite{Li-FMC}.

\begin{prpt} There is a $(\Gamma_q,1)$-vertex algebra structure on $V_{\wh{H}}$
such that $R_{\lambda}\(b^{\alpha}\)=b^{\alpha\lambda^{-1}}$ for $\alpha,\lambda\in \Gamma_q$.
Furthermore, $\Gamma_q$-equivariant $\phi_1$-coordinated quasi $V_{\wh{H}}$-modules are exactly restricted $H_q$-modules.
\end{prpt}

\begin{remt} Similar to the Lie algebras $\wh{\CL}(\mathfrak{b},\sigma)$ and $\wh{\fgl_\ell}(\C_Q)$,
for any integer $\epsilon\in \Z$,  $H_q$ is a quasi  vertex  Lie algebra with the generating functions $\mb{c}(z)=\mb{c}z^{\epsilon-1}$ and $a(z)=\sum_{m\in\Z}a(m)z^{-m+\epsilon-1}$.
Recall that for $\g=\wh{\CL}(\mathfrak{b},\sigma)$ (resp. $\wh{\fgl_\ell}(\C_Q)$), $\g^0$ is isomorphic to $\wh{\CL}(\mathfrak{b})$ (resp. $\fgl_{\ell, Q}$) for any $\epsilon\in\Z$.
However, for the Lie algebra $H_q$,  $H_q^0$ has infinite-dimensional center when $\epsilon\neq1$, while it has one-dimensional center when $\epsilon=1$.
\end{remt}

\subsection{Virasoro-like algebras} In this subsection, we consider the Virasoro-like algebra
 \[\mathcal{VL}=\oplus_{m_1,m_2\in \Z} \C L_{m_1,m_2}\oplus \C \mb{c},\]
where $\mb{c}$ is a central element and for any $m_1,m_2,n_1,n_2\in\Z$,
\begin{align}\label{lb-virasoro-fenliang}
[L_{m_1,m_2},L_{n_1,n_2}]=(m_1n_2-m_2n_1)L_{m_1+n_1,m_2+n_2}+\delta_{m_1,-n_1}\delta_{m_2,-n_2}m_1\mb{c}.
\end{align}
Set $L_m(z)=\sum_{n\in\Z}L_{n,m}z^{-n}$ for $ m\in \Z$.
Then \eqref{lb-virasoro-fenliang} is equivalent to:
\begin{align*}
[L_m(z),L_n(w)]
=&(m+n)L_{m+n}(w) \(w\frac{\partial}{\partial w}\) \delta\left(\frac{w}{z}\right)+m\(w\frac{\partial}{\partial w}L_{m+n}(w)\) \delta\left(\frac{w}{z}\right) \\
&+\delta_{m,-n}\mb{c} \(w\frac{\partial}{\partial w}\) \delta\left(\frac{w}{z}\right).
\end{align*}
It follows that  $(\mathcal{VL},\CA,1)$ is a maximal quasi vertex Lie algebra with the trivial group $\{1\}$ as the associated group, where $$\CA=\{nL_m(z),(n\mb{c})(z):=n\mb{c}\mid m,n\in\Z\}.$$
And for any $\zeta\in \Z$, the Lie algebra $\mathcal{VL}^{\zeta}$ (see Theorem \ref{thm:main1}) admits a basis
 \[L_m^{1,\zeta}(n),\ \mb{c}^{1,\zeta}(\zeta-1)\quad\te{ for } m,n\in\Z,\]
such that $\mb{c}^{1,\zeta}(\zeta-1)$ is central and for $m,n,m',n'\in\Z$,
\begin{align*}
&\,[L_{m}^{1,\zeta}(m'-\zeta+1),L_n^{1,\zeta}(n'-\zeta+1)]\\
=&\,\((m'-\zeta+1)n-m(n'-\zeta+1)\)L_{m+n}^{1,\zeta}(m'+n'-\zeta+1)\\
&\,+\delta_{m+n,0}\delta_{m'+n'-2(\zeta-1),0}(m'-\zeta+1)\mb{c}^{1,\zeta}(\zeta-1).
\end{align*}

We consider a variant of Lie algebra $\mathcal{VL}$ as follows:
\[\mathcal{VL}'=\oplus_{m_1,m_2\in \Z} \C L_{m_1,m_2}'\oplus \C \mb{c}',\]
where $\mb{c}'$ is a central element and for any $m_1,m_2,n_1,n_2\in\Z$,
\begin{align*}
[L_{m_1,m_2}',L_{n_1,n_2}']=((m_1+1)n_2-m_2(n_1+1))L_{m_1+n_1,m_2+n_2}'+\delta_{m_2+n_2,0}\delta_{m_1+n_1+2,0}(m_1+1)\mb{c}'.
\end{align*}
We see that  the Lie algebra $\mathcal{VL}'$ is isomorphic to $\mathcal{VL}^0$ with
$L'_{m_1,m_2}\mapsto L_{m_2}^{1,0}(m_1+1)$ and $\mb{c}'\mapsto \mb{c}^{1,0}(-1)$.
Let $\C$ be the trivial $\mathcal{VL}'_+=\sum_{m_1\ge -1,m_2\in \Z}\C L_{m_1,m_2}'$-module and form the induced module
\[ V_{\mathcal{VL}'}=\U(\mathcal{VL}')\otimes_{\U(\mathcal{VL}'_+)}\C.\]
Set ${\bf 1}=1\otimes 1$, $L'_m=L'_{-2,m}\ot 1$ and $L_m'(z)=\sum_{n\in \Z}L_{n,m}'z^{-n-2}$ for $m\in \Z$. Then from the
 Theorem \ref{thm:main2} we have the following result, which was also obtained in \cite{BLP}.

\begin{prpt} There is a vertex algebra structure on $V_{\mathcal{VL}'}$ with $Y(L_m',z)=L_m'(z)$ for $m\in \Z$.
Furthermore, $\phi_1$-coordinated  $V_{\mathcal{VL}'}$-modules are exactly  restricted $\mathcal{VL}$-modules.
\end{prpt}

\subsection{Klein bottle Lie
algebras} We consider the involution $\sigma$ of $\mathcal{VL}$ defined by
$$\sigma(\mb{c})=\mb{c},\quad \sigma(L_{m_1,m_2})=-(-1)^{m_1}L_{m_1,-m_2} \quad \te{for }m_1,m_2\in\Z.$$
Denote by $\mathcal{B}$ the $\sigma$-fixed point subalgebra of $\mathcal{VL}$, which is a one-dimensional central extension of
the Klein bottle Lie algebra (\cite{JJP,PR}).
 Set $$B_{m_1,m_2}=L_{m_1,m_2}-(-1)^{m_1}L_{m_1,-m_2}\quad\text{for}\ m_1,m_2\in \Z,$$
which together with $\mb{c}$ span the Lie algebra $\mathcal{B}$.
Note that we have $B_{m_1,m_2}=-(-1)^{m_1}B_{m_1,-m_2}$, and
\begin{align*}
[B_{m_1,m_2},B_{n_1,n_2}]
=&(m_1n_2-m_2n_1)B_{m_1+n_1,m_2+n_2}-(-1)^{m_1}(m_1n_2+m_2n_1)B_{m_1+n_1,n_2-m_2}\\
&+2(\delta_{m_2,-n_2}-(-1)^{m_1} \delta_{m_2,n_2})\delta_{m_1,-n_1}m_1\mb{c}
\end{align*}
for $m_1,m_2,n_1,n_2\in\Z$. In terms of the generating functions $B_m(z)=\sum_{n\in\Z}B_{n,m}z^{-n}$ for $m\in\Z$, we have
  $B_m(-z)=-B_{-m}(z)$ and
\begin{align*}
\,[B_m(z),B_n(w)]
=&\,(m+n)B_{m+n}(w) \(w\frac{\partial}{\partial w}\) \delta\left(\frac{w}{z}\right) +m\(w\frac{\partial}{\partial w}B_{m+n}(w)\)  \delta\left(\frac{w}{z}\right) \\
&+(m-n)B_{n-m}(w) \(w\frac{\partial}{\partial w}\) \delta\left(\frac{-w}{z}\right)+m\(w\frac{\partial}{\partial w}B_{n-m}(w)\) \delta\left(\frac{-w}{z}\right) \\
&+2\delta_{m,-n}\mb{c} \(w\frac{\partial}{\partial w}\) \delta\left(\frac{w}{z}\right) -2\delta_{m,n}\mb{c} \(w\frac{\partial}{\partial w}\) \delta\left(\frac{-w}{z}\right).
\end{align*}
Then $(\mathcal{B},\CA,1)$ is a maximal quasi vertex Lie algebra with $\Gamma_\sigma=\{\pm 1\}$ as the associated group, where
$$\CA=\{nB_m(z), (n\mb{c})(z):=n\mb{c}\mid m,n\in\Z\}.$$

  Note that the following relations hold in $\mathcal{B}^\zeta$ ($\zeta\in\Z$):
\[\mb{c}^{\pm1,\zeta}(z)=\mb{c}^{1,\zeta} (\zeta-1),\quad B_m^{-1,\zeta}(z)=-B^{1,\zeta}_{-m}(z)\quad\te{for }m\in\Z.\]
This implies that the Lie algebra $\mathcal{B}^{\zeta}$ has a basis $\mb{c}^{1,\zeta} (\zeta-1)$, $B_m^{1,\zeta}(n)$ for
$ m,n\in\Z$.  It is straightforward to check that
 the Lie algebra $\mathcal{VL}^\zeta$ is isomorphic to $\mathcal{B}^\zeta$ with  the isomorphism given by  $\mb{c}^{1,\zeta}(\zeta-1)\mapsto2\mb{c}^{1,\zeta}(\zeta-1)$ and $L_m^{1,\zeta}(n)\mapsto B_m^{1,\zeta}(n)$ for $m,n\in\Z$.
 In particular, when $\zeta=0$, we have $\mathcal{B}^0\cong \mathcal{VL}'$.
From the Theorem \ref{thm:main2} we immediately have the following result.

\begin{prpt} There is a $(\Gamma_\sigma, 1)$-vertex algebra structure on $V_{\mathcal{VL}'}$ with $R_{-1}\(L_m'\)=-L_{-m}'$
for $m\in\Z$. Furthermore,
$\Gamma_\sigma$-equivariant $\phi_1$-coordinated quasi $V_{\mathcal{VL}'}$-modules are exactly restricted $\mathcal{B}$-modules.
\end{prpt}

\begin{remt} Recall that if $\g=\wh{\CL}(\mathfrak{b},\sigma)$ or $\wh{\fgl_\ell}(\C_Q)$,  we have $\g^\zeta\cong \g^0$ for any $\zeta\in \Z$.
 However, when $\g=\mathcal{VL}$ or $\mathcal{B}$, it is known that $\g^1\cong\mathcal{VL}$ which is not isomorphic to $\g^0
\cong\mathcal{VL}'$ (see \cite{DZ}).
\end{remt}

\begin{remt} We note that the Virasoro-like algebras and the Klein bottle Lie algebras  are not quasi vertex Lie algebras if we write the generating functions as $\sum_{n\in \Z} L_{n,m} z^{\epsilon-n-1}$
and $\sum_{n\in \Z} B_{n,m} z^{\epsilon-n-1}$  for $m\in \Z$ respectively, unless $\epsilon=1$.
The similar phenomenons appear in the generating functions of the (twisted)  toroidal extended affine Lie algebras (see \cite{CLT,CTY}).
\end{remt}

\end{document}